\documentclass[10pt]{amsart}

\usepackage{amsmath,amssymb,amsfonts,euscript,enumerate,amsthm}
\usepackage{color}

\usepackage{graphicx}
\usepackage{epstopdf}
\usepackage[colorlinks=true,citecolor=red,linkcolor=blue]{hyperref}
\usepackage[]{cleveref}
\usepackage{mathtools}
\usepackage{esint}
\usepackage{mathrsfs}
\usepackage[notref, notcite, final]{showkeys}

\newtheorem{theorem}{Theorem}[section]
\newtheorem{lemma}[theorem]{Lemma}

\newtheorem{corollary}[theorem]{Corollary}

\theoremstyle{plain}
\newtheorem*{namedthm}{\namedthmname}
\newcounter{namedthm}
\makeatletter
\newenvironment{named}[1]
  {\def\namedthmname{#1}%
   \refstepcounter{namedthm}%
   \namedthm\def\@currentlabel{#1}}
  {\endnamedthm}
\makeatother

\theoremstyle{definition}
\newtheorem{definition}[theorem]{Definition}

\theoremstyle{remark}

\numberwithin{equation}{section}

\newcommand{\thistheoremname}{}
\newtheorem*{genericthm*}{\thistheoremname}
\newenvironment{namedthm*}[1]
  {\renewcommand{\thistheoremname}{#1}%
   \begin{genericthm*}}
  {\end{genericthm*}}

\newcommand*{\dd}{\mathop{}\!\mathrm{d}}

\newcommand{\R}{{\mathbb R}}

\newcommand{\cF}{{\mathcal F}}
\newcommand{\cH}{{\mathcal H}}

\newcommand{\cP}{{\mathcal P}}

\newcommand{\al}{\alpha}

\newcommand{\ga}{\gamma}
\newcommand{\de}{\delta}
\newcommand{\e}{\varepsilon}

\newcommand{\la}{\lambda}
\newcommand{\si}{\sigma}
\newcommand{\vp}{\varphi}
\newcommand{\om}{\omega}

\newcommand{\Om}{\Omega}
\newcommand{\La}{\Lambda}

\newcommand{\Ga}{\Gamma}

\newcommand{\ti}{\times}

\newcommand{\pa}{\partial}
\newcommand{\su}{\subset}
\newcommand{\qu}{\quad}

\newcommand{\sm}{\setminus}
\newcommand{\ra}{\rightarrow}

\newcommand{\D}{\nabla}
\newcommand{\De}{\Delta}

\newcommand{\dist}{\operatorname{dist}}
\newcommand{\supp}{\operatorname{supp}}
\newcommand{\dv}{\operatorname{div}}

\newcommand{\trm}{\textrm}
\newcommand{\ov}{\overline}
\newcommand{\fr}{\frac}

\newcommand{\norm}[1]{\left\lVert#1\right\rVert}

\newcommand{\2}{\frac{1}{2}}

\def\({\left(}
\def\){\right)}
\def\<{\left\langle}
\def\>{\right\rangle}

\newcommand{\fc}{f\chi_{D}-g\chi_{D^c}}

\begin{document}

\title[Nonlocal composite membrane problem]{Optimal configuration and symmetry breaking phenomena in the
composite membrane problem with fractional Laplacian}

\author{Mar\'ia del Mar Gonz\'alez}
\address{Universidad Aut\'onoma de Madrid. Departamento de Matem\'aticas, Campus de Cantoblanco, 28049
Madrid, Spain}
\email{mariamar.gonzalezn@uam.es}

\author{Ki-Ahm Lee}
\address{Department of Mathematical Sciences, Seoul National University, Seoul 08826, Korea \& Korea
Institute for Advanced Study, Seoul 02455, Korea}
\email{kiahm@snu.ac.kr}

\author{Taehun Lee}
\address{Department of Mathematical Sciences, Seoul National University, Seoul 08826, Korea}
\email{boytaehun@snu.ac.kr}



\begin{abstract}	
We consider the following eigenvalue optimization in the composite membrane problem with fractional
Laplacian: given a bounded domain $\Omega\subset \mathbb{R}^n$, $\alpha>0$ and $0<A<|\Om|$, find a
subset $D\subset \Omega$ of area $A$ such that the first Dirichlet eigenvalue of the operator
$(-\Delta)^s+\alpha \chi_D$ is as small as possible. The solution $D$ is called as an optimal
configuration for the data $(\Omega,\alpha,A)$. Looking at the well-known extension definition for the fractional Laplacian, in the case $s=1/2$ this is essentially the composite membrane problem for which the mass is concentrated at the boundary as one is trying to minimize the Steklov eigenvalue.

We prove existence of solutions and study properties
of optimal configuration $D$. This is a free boundary problem which could be formulated as a two-sided unstable obstacle problem.

Moreover, we show that for some rotationally symmetric domains (thin
annuli), the optimal configuration is not rotational symmetric, which implies the non-uniqueness of
the optimal configuration $D$. On the other hand, we prove that for a convex domain $\Omega$ having
reflection symmetries, the optimal configuration possesses the same symmetries,
which implies uniqueness of the optimal configuration $D$ in the ball case.
\end{abstract}

\maketitle

\section{Introduction}
We study an eigenvalue optimization in the composite membrane problem with fractional Laplacian. Let
$\Om$ be a bounded domain in $\R^n$ with $C^{1,1}$-boundary and $D\su \Om$ be a measurable subset. For
$0<s<1$ and $\al>0$, we consider the following eigenvalue problem
\begin{equation*}\label{eq-ev}
\begin{alignedat}{3}
	(-\Delta)^s u+\al \chi_Du&=\lambda u &&\quad\textrm{ in }\Omega,\\
	u&=0 &&\quad\textrm{ on }\R^n\setminus\Omega.
\end{alignedat}
\end{equation*}
Denote $\lambda_\Omega(\al,D)$ by the first eigenvalue of \eqref{eq-ev} and, for $0\le A\le |\Om|$,
define
\begin{align}\label{Lala}
\La_\Om(\al, A)=\inf_{|D|=A} \lambda_\Omega(\al, D).
\end{align}
If $D$ attains the minimum of \eqref{Lala}, then we call $D$ as an optimal configuration for the data
$(\Om,\al,A)$ and $(u,D)$ as an optimal pair. The main objective of this paper is to study optimal pairs
of the nonlocal problem above.

\begin{named}{Problem $\mathbf{(N)}$}\label{pb:N}
Investigate:
\begin{enumerate}[(i)]

\item the existence and regularity of optimal pairs $(u,D)$,

\item the shape of optimal configurations, $D$,

\item and the uniqueness of $(u,D)$ (up to a multiplicative constant for $u$).

\end{enumerate}

\end{named}

\subsection{History}
The composite membrane problem for the Laplacian operator was considered in
\cite{CGI+} which shows existence and regularity of optimal pairs. Moreover, it was shown in the same
paper that the optimal configuration $D$ is given by a sublevel set of $u$ whenever $(u,D)$ is an
optimal pair, i.e.,
\begin{align*}
D=\{x\in \Om: u(x)\le t\}
\end{align*}
for some constant $t$ satisfying $|D|=A$. In addition, they obtained symmetry and symmetry breaking
phenomena of $D$, which imply uniqueness and non-uniqueness depending on the domain $\Om$.

Following this work, the optimal regularity of optimal pairs and the regularity and singularity of the
free boundary $\pa D$ have been studied by several authors in \cite{CGK,Shahgholian07,CK,CKT08}. In
particular, the optimal regularity in dimension two was shown in \cite{CKT08}. We also refer to \cite{Pieli, CEP, AC} for the $p$-Laplacian version of composite membrane problem.

In addition, the very  recent works \cite{CV,CV2} consider the bi-Laplacian case, which is related to the composite plate
problem. They established a symmetry property which implies the uniqueness of the optimal pairs for the
corresponding problem when the domain is a ball, and the existence of the optimal pairs. However,
symmetry breaking phenomena, whose direct consequence is non-uniqueness, have not been considered except
for the Laplacian case \cite{CGI+}, even though some numerical evidence supports the occurrence of
symmetry breaking phenomena (see \cite{CV2} and \cite{KK}).

Finally, we recall the work \cite{Chanillo:conformal}, which relates the composite membrane problem to a
certain eigenvalue minimization problem in
two dimensions for the Laplace operator in conformal classes. They also provide the generalization to
any even dimension $n$, where the Laplacian is replaced by the GJMS operators (these are conformally covariant operators which the same principal symbol as $(-\Delta)^{n/2}$). For odd dimensions, this equivalence should be also possible  and the natural
setting is that of fractional order operators.

\subsection{Main
results}
Let us now go back to our question, \ref{pb:N}. We first note that $\la_\Om$ is invariant under any
change of $D$ by a measure zero set, so we will ignore such differences. Also,
we assume that
\begin{align}\label{assumption-alpha}
\al\le \ov\al_\Om(A),
\end{align}
where $\ov\al_\Om(A)$ is the unique constant satisfying $\La_\Om(\ov \al_\Om(A),A)=\ov \al _\Om(A)$. The
convenience of this notation will be clear in Lemma \ref{PN relation}.

Our first result concerns the existence of an optimal pair and properties of optimal configurations,
which give answers to $(i)$ and $(ii)$ in \ref{pb:N}. However, in contrast to the local case considered in
\cite{CGI+}, it is nontrivial to show that the optimal configuration $D$ is given by a sublevel set of
$u$ due to nonlocal effects. The reason for such difficulty comes from the fact that $(-\De)^s u$ may not be zero at a point where $u$ is
locally constant. In any case, it is quite straightforward to prove that
\begin{align*}
\{x\in\Om:u(x)<t\} \su D\su \{x\in\Om:u(x)\le t\},
\end{align*}
where $t:=\sup \{c:|\{u<c\}|<A\}$.

Now we state our first main result:

\begin{theorem}\label{thm-reg}
Let $\al>0$ satisfying \eqref{assumption-alpha}  and $A\in [0,|\Om|]$. Then:
 \begin{itemize}
 \item[\emph{(i)}] There exists an optimal pair $(u,D)$.
 \item [\emph{(ii)}] Any optimal pair satisfies
\begin{align*}
u\in H^{2s}_{loc}(\Om) \cap C^{\beta} (\Om) \cap C^s(\R^n),
\end{align*}
where $\beta$ is $2s$ if $s\not=\2$, and any constant in $(0,2s)$ if $s=\2$.
\item[\emph{(iii)}] Let $\al<\ov \al_\Om(A)$. If $s\le \2$, the optimal configuration $D$ is a sublevel set of $u$, i.e.,
\begin{align}\label{thmeq-D}
D=\{x\in \Om:u(x)\le t\},
\end{align}
where $t:=\sup \{c:|\{u<c\}|<A\}$.
\end{itemize}
\end{theorem}

Notice that $C^{2s}(\Om)$ (resp. $C^{0,1}(\Om)$) for $s\not=\fr{1}{2}$ (resp. $s=\fr{1}{2}$) is the optimal regularity for $u$ since $(-\De)^s u$ is not continuous in
$\Om$. We also remark that the sublevel set property \eqref{thmeq-D} for local operators can be easily
proved, if one has sufficient regularity, from the well-known fact that the weak derivative of $u$ is
zero a.e. on its constant set (see \cite{CGI+} for the Laplacian and \cite{CV} for the bi-Laplacian).
But this property no longer holds for nonlocal operators. This is also the case in the $p$-Laplacian
version of composite membrane problem because of the lack of regularity \cite{Pieli}.

However, we can still expect the sublevel set property \eqref{thmeq-D} even though $(-\De)^su$ may not
be zero on the locally constant points. Heuristically,
this is because
$(-\De)^s u$ is continuous at the locally constant points (\Cref{lem-continuity}) so that any connected
component of the interior of $\{u=t\}$ should be contained in either $int(D)$ or $int(\Om\sm D)$, see
\Cref{cor-D}.\\

In order to show the sublevel set property, we need to borrow some techniques coming from free boundary problems. The main idea is, first, to adapt the arguments in \cite{FF} on unique continuation properties for fractional Laplacian equations in order to show that the level set $\{u=t\}$ has measure zero. This involves looking at the structure of blowup limits at a free boundary point, proving that they are non-trivial in order to get a contradiction.

In particular, for every $s\in(0,1)$ we prove that blowup sequences converge.
Then, to show that their limits are non-trivial, we study optimal regularity and non-degeneracy.  Our proof only works when $s\in(0,1/2]$ because we are only able to control non-degeneracy  in this case. More precisely, see the proof of Lemma \ref{lem-nondeg}, which is based on the arguments in \cite{CRS}. Note also that the case $s=\2$ is more involved due to the loss of regularity and it needs to be considered separately in the proof.

One of the crucial steps to obtain regularity is to transform our problem into a
two-phase unstable obstacle problem for the fractional Laplacian. Defining
$v=t-u$, $f=(\La-\al )u$ and $g=-\La u$, we consider
\begin{equation}\label{obs-int}
\begin{split}
	-(-\Delta)^s u&= f\chi_{D}-g\chi_{D^c} \qu\textrm{on } \Om\subset \R^{n},\\
	v&=t \quad \qu\textrm{on } \mathbb R^n\setminus \Om.
\end{split}
\end{equation}
In our case, $f$ and $g$ are functions with
\begin{equation}\label{conditions:fg}
f>0, \quad g<0, \quad f+g<0,
\end{equation}
 which is referred as an unstable problem, see \cite{MW,AG}.

The classical version of \eqref{obs-int} has been studied earlier for various conditions on $f$ and $g$. For
instance, if $f>0$ and $g>0$, the corresponding two-phase membrane problem was
considered in \cite{SW06, SUW07, SUW04, LSE09}. In the case of $f>0$ and $f+g>0$, on the other hand, we
refer to \cite{Weiss01, Uraltseva01}. The composite membrane problem, which corresponds to conditions \eqref{conditions:fg}, can be found in \cite{CGI+, Shahgholian07, CK, CKT08} as stated above.
However, to the best of authors knowledge, the nonlocal version \eqref{obs-int}, has not been studied so far
except \cite{ALP,AG} whose $f$ and $g$ are constants. \\

Our second result shows that if the domain has some geometric properties, then optimal pairs also have
some geometric properties. The proof is based on the Steiner symmetrization method with a slight
modification of the kernel in the singular integral.

\begin{theorem}\label{thm-symm}
Let $\Om$ be a domain in $\mathbb R^n$. Assume that it has symmetry and convexity with
respect to the hyperplane $\{x_1=0\}$, i.e., for each $x'\in\R^{n-1}$ the set $\{x_1:(x_1,x')\in\Om\}$
is either an interval of the form $(-b,b)$ or the empty set. Then, for any optimal pair $(u,D)$ both $u$
and $D$ are symmetric with respect to $\{x_1=0\}$, and $D^c$ is convex with respect to
$\{x_1=0\}$. Moreover, $u$ is decreasing in $x_1$ for $x_1 \geq 0$.
\end{theorem}
Using this symmetry property, we obtain the first uniqueness result when the domain is a ball. Indeed, once we
have that the optimal configuration is of the form \eqref{thmeq-D}, $D$ is determined by $A$, which
makes that our problem does not have a free boundary anymore.  For this fixed boundary problem, we can
see that $u$ is a unique solution by the strong maximum principle and the fact that any solution has
a sign, i.e., either $u$ is always positive or $u$ is always negative.

\begin{corollary}\label{cor-ball}
Assume that the domain $\Om$ is the unit ball. Then there is a unique optimal pair $(u,D)$ (up to
multiplication by nonzero constants). Moreover, the solution $u$ is rotationally symmetric and strictly
decreasing in radial direction, and $D$ is a shell region of the form
\begin{align}\label{D-radial}
D=\{x:r(A)\le |x|<1\},
\end{align}
where $r(A)$ is the constant satisfying $|D|=A$.
\end{corollary}

Our last result in this paper is a symmetry breaking property when the domain is an annulus in $\R^2$,
for $0<s<\frac 1 2$. While the scheme of proof follows closely that of the local case in \cite{CGI+}, the arguments in their paper use the fact that the Laplacian has a simple expression  in polar coordinates. On the contrary, the fractional Laplacian has no easy decomposition in spherical harmonics (see, for instance, \cite{ACDFGW} and the references therein).

One of the main ingredients in our proof is the following decomposition formula: in
polar coordinates, for any $x=(r,\theta_0)\in \Om$ with fixed angle, given a function of the form
$v=f(r)g(\theta)$,
\begin{align*}
(-\De)^s v(x)= g(\theta_0)(-\De)^s f (r)+c_{n,s}\int_0^\infty\int_{\pa
B_{\rho}}\fr{f(\rho)(g(\theta_0)-g(\theta))\rho^{n-1}}{|x-(\rho,\theta)|^{n+2s}} \,\trm{d} \theta\trm{d}
\rho.
\end{align*}
Moreover, we relate the fractional Laplacian $(-\De)^s$ of a rotationally symmetric function to the one-dimensional fractional Laplacian $(-\De)^s_1$ of a function of radial variable.
This connects some rotationally symmetric eigenvalue problem with a third eigenvalue problem which is
not rotationally symmetric. See \Cref{lem-estiB} and the equations above for details. We remark that
the annulus is symmetric with respect to any axis but it is not convex, which violates the assumption in
\Cref{thm-symm}.

\begin{theorem}\label{thm-symm-breaking}
Let $0<s<\2$. For an annulus domain \begin{align*}
\Om_b=\{x\in \R^2: b<|x|<b+1\}
\end{align*} with sufficiently large $b\ge 1$, the optimal configuration $D$ in $\Om_b$ does not have
rotational symmetry.
\end{theorem}

This yields non-uniqueness of the optimal pair since any rotation of a solution without symmetry
generates a new solution. In conclusion, the uniqueness issue $(iii)$ in \ref{pb:N}, is that the
optimal pair is not unique in general, at least for $0<s<\2$, but it is so for some special domains.\\

We close this subsection with some remarks on the free boundary regularity. When $s>1/2$, one can easily obtain that the free boundary near a regular point (i.e. $Du\not=0$) is locally a $C^1$ graph by using the implicit function theorem. To obtain a similar result near a singular point (i.e. $Du=0$), one may need singularity analysis. Nevertheless, if $s\le1/2$, one can not expect $C^1$ regularity near the free boundary, which means that we can not use the implicit function theorem. The expected regularity for free boundary is smooth or analytic, which is left for the future work.

\subsection{Equivalent
problems}
Let us discuss two equivalent formulations of \ref{pb:N}. The first is a local expression via the
well-known Caffarelli-Silvestre extension \cite{CS07}. Let $a=1-2s$ so that $-1<a<1$. For a bounded
domain $\Om \subset \R^n$ with $C^{1,1}$ boundary and numbers $\al >0 $, $A\in [0,|\Om|]$, we consider
the eigenvalue problem with mixed boundary conditions
\begin{equation}\label{EL-E}
\begin{alignedat}{3}
	L_a u&=0 &&\quad \textrm{ in } \R^{n+1}_+,\\
	-M_au+\al \chi_D u &= \la u&&\quad \textrm{ on } \Om\ti\{0\}\subset \pa \R^{n+1}_+,\\
	u&=0&& \quad \textrm{ on } (\R^n\setminus \Om)\ti \{0\}\subset \pa \R^{n+1}_+,
\end{alignedat}
\end{equation}
where $D\subset \Om$ is any measurable subset in $\R^n$, and the operators $L_a$ and $M_a$ are the usual
ones in the extension to $\mathbb R^{n+1}_+$ for the fractional Laplacian and are defined in
\eqref{notation-operators}.

Denote  the first eigenvalue by $\la_\Om(\al,D)$ and define
\begin{align}\label{Lala-E}
	\La_\Om(\al, A)=\inf _{|D|=A} \la_\Om (\al, D).
\end{align}
If $D$ attains the minimum of \eqref{Lala-E}, then we call it as an optimal configuration for the data
$(\Om,\al,A)$ and $(u,D)$ an optimal pair.

\begin{named}{Problem $\mathbf{(E)}$}\label{pb:E}
Investigate the same questions as in \ref{pb:N} for \eqref{EL-E} and \eqref{Lala-E}.
\end{named}

From the well known extension theorem for the fractional Laplacian \cite{CS07} we know that
\begin{align*}
M_a u:= \lim_{y\to 0^+} y^a u_y = -C_{n,s}(-\De)^s\big(u(\cdot,0)\big).
\end{align*}
Thus we can identify \ref{pb:N} and \ref{pb:E}. We shall use this equivalence in
\Cref{sec:symm-breaking} about symmetry breaking phenomena.\\

We will consider two more equivalent forms of \ref{pb:E}. First, defining $v=t-u$, we can consider problem \eqref{eq-v}, which is the extension version of \eqref{obs-int}. This will be useful when we deal with blowups in \Cref{sec:blowups}.

We will also consider the formulation as a physical problem  (\ref{pb:P1} in \Cref{sec:phy}). We will
see that there is a unique constant $\ov\al_\Om(A)$ satisfying $\La_\Om(\ov \al_\Om(A),A)=\ov \al
_\Om(A)$ so that if $\al\le \ov\al_\Om(A)$, then \ref{pb:P1} is equivalent to \ref{pb:N}. From this
relation\textcolor{cyan}{,} we can give the physical interpretation of our problem, which is to optimize the basic frequency of an elastic membrane whose boundary has two parts; one is fixed and another is free but a prescribed mass is
concentrated on the latter. This basic frequency is essentially given by the Steklov eigenvalue of the membrane.

\subsection{Outline}
In \Cref{sec:pre} we introduce notations, definition of weak solutions, and some properties for the
solution to the extension problem. Also, in \Cref{sec:phy} we discuss the physical interpretation so that
\ref{pb:N} contains \ref{pb:P1}, while in \Cref{sec:basic} we prove the existence of optimal pairs and
some regularity results, which yield some of the statements in \Cref{thm-reg}. The non-triviality of blowup limits is
considered in \Cref{sec:blowups} and \Cref{sec:nondeg}, and then we complete the proof of
\Cref{thm-reg} in \Cref{sec:struc}. The proofs of \Cref{thm-symm} and \Cref{cor-ball} are given in
\Cref{sec:symm}. Finally, we discuss the symmetry breaking phenomena in \Cref{sec:symm-breaking} and
hence we prove \Cref{thm-symm-breaking}.

\section{Preliminaries}\label{sec:pre}
\subsection{Notations}
The following notations are used throughout the paper:
\begin{enumerate}
\item Let $|\cdot|$ and $\dd\si$ denote, respectively, Lebesgue measure and the surface measure. The
    outer unit normal vector is denoted by $\nu$.
\item Let $\Om$ be a bounded domain with $C^{1,1}$ boundary unless otherwise specified.
\item Let $\mu_\Om$ be the first eigenvalue in $\Om$ for the fractional Laplacian operator with zero
    Dirichlet boundary conditions on $\Omega^c$.
\item For a constant $1<\ga<2$, the space $C^\ga$ is understood as $C^{1,\ga-1}$.
\item Denote $\mathcal{F}$ by the free boundary $\pa D$. 
\item For a point $X$ in $\R^{n+1}$ we often use $x$ to denote the first $n$ coordinates and $y$ for
    the last coordinate so that $X=(x,y)$.
\item For balls and half balls:
\begin{equation*}
\begin{split}
&B_R(X_0) = \{X\in \R^{n+1}: |X-X_0|< R\},\\
&B_R^+(X_0) = B_R(X_0)\cap \{(x,y)\in\R^{n+1}:y>0\},\\
&\Ga_R^+=\pa B_R\cap \{(x,y)\in\R^{n+1}:y\ge 0\},\\
&\Ga_R^0= \{(x,0)\in\R^{n+1}:|x|<R\}.\\
\end{split}
\end{equation*}
\item We define the operators on $\mathbb R^{n+1}_+$, for $a=1-2s$,
\begin{equation}\label{notation-operators}
\begin{split}
&L_av:= \dv(y^a \D v),\\
&M_av:=\lim_{y\ra 0} y^av_y.
\end{split}
\end{equation}
\item The weighted Lebesgue space $L^2(B,y^a)$ is a Banach space with the norm
\begin{align*}
\norm{f}_{L^2(B,y^a)}=\(\int_B |f(x)|^2y^a \dd X <\infty\)^{1/2}.
\end{align*}
The weighted Sobolev space $H^1(B,y^a)$ is defined similarly. If $a=0$, the spaces  $L^2(B)$
and $H^1(B)$ are defined in the usual way.
\end{enumerate}

\subsection{Fractional spaces and weak solutions}
Let $0<s<1$, and $\Om\su \R^n$ be a domain. The fractional Laplacian $(-\De)^s $ on
$\mathbb R^n$ is defined as
\begin{align*}
(-\De)^s u(x)= c_{n,s} \int _{\R^n} \fr{u(x)-u(y)}{|x-y|^{n+2s}}\dd y,
\end{align*}
where $c_{n,s}$ is a normalization constant. We also introduce the classical fractional Sobolev space $H^s(\Om)$ defined by
\begin{align*}
H^s(\Om)= \left\{ u \in L^2 (\Om): \fr{u(x)-u(y)}{|x-y|^{\fr{n+2s}{2}}}\in L^2(\Om\ti
\Om)\right\},
\end{align*}
endowed with the norm
\begin{align*}
\norm{u}_{H^s(\Om)}= \norm{u}_{L^2(\Om)}+ \left(\int_{\Om\ti\Om}
\fr{|u(x)-u(y)|^2}{|x-y|^{n+2s}}\dd x \dd y\right)^\2.
\end{align*}
  The localized version of $H^s(\Om)$
is denoted by
\begin{align*}
H^{s}_{loc}(\Om)= \left\{ u\in L^2(\Om): u\eta \in H^s(\Om)\trm{ for any test function }\eta\in \mathcal{D}(\Om) \right\},
\end{align*}
where $\mathcal{D}(\Om)$ denotes the space of all continuously infinitely differentiable functions with compact support in $\Om$.

The admissible set for  weak solutions for our non-local equation is given by
\begin{align*}
H^s_0(\Om) = \left\{ u \in H^s(\R^n) : u \equiv 0 \trm{ in } \R^n\sm \Om\right\}.
\end{align*}
Weak solutions are then defined as follows: for any bounded function $\rho: \R^n \ra \R$, a function $u$ is called a weak solution of
\begin{equation*}
\begin{split}
(-\De)^s u +\rho u &=0 \qu \trm{in } \Om,\\
u &= 0 \qu \trm{in } \R^n\sm \Om
\end{split}
\end{equation*}
 if $u\in H^s_0(\Om)$, and
\begin{align*}
\fr{c_{n,s}}{2}\int_{\R^n\ti \R^n}\fr{(u(x)-u(y))(\vp(x)-\vp(y))}{|x-y|^{n+2s}} \dd x \dd y + \int_{\Om}
\rho(x)u(x)\vp(x) \dd x=0.
\end{align*}
for any $\vp \in  H^s_0(\Om)$.

\subsection{Extension problem}

In this subsection, we list some properties for the extended function from $\R^n$ to $\R^{n+1}_+$.

We recall the well-known Caffarelli-Silvestre extension (see \cite{CS07} and also \cite{CS14}). Given a function $u=u(x)$ on $\mathbb R^n$, its extension to $\mathbb R^{n+1}_+$ (still denoted by the same letter $u=u(x,y)$)  is the solution to
\begin{equation}\label{CS-extension}
\begin{split}
L_a u&=\De _x u+\fr{a}{y}u_y+u_{yy}=0 \quad \trm{in} \quad \R^{n+1}_+ ,\\
u(x,0)&=u(x) \quad \trm{on} \quad \R^n.
\end{split}
\end{equation}
The extended function can be also written as
\begin{align*}
u(x,y)= (P(\cdot ,y) * u)(x),
\end{align*}
where $P$ is the Poisson kernel
\begin{align}\label{Poisson}
P(x,y)=C_{n,a}\fr{y^{1-a}}{(|x|^2+y^2)^{\fr{n+1-a}{2}}},
\end{align}
with the constant $C_{n,a}$ chosen such that $\int_{\R^n} P(x,y)\dd x=1$. It is well known that
\begin{equation*}
(-\Delta)^s u(x)=d_s M_a u \qu \trm{in } \R^n=\pa \R^{n+1}_+,
\end{equation*}
where $d_s=2^{2s-1}\Ga(s)/\Ga(1-s)$.
Next, we give the definition of (localized) weak solutions for the extension problem:

\begin{definition}[Weak solutions]\label{def-weak-local}
Let $-1<a<1$, $r>0$, and $h\in L^1(\Ga_r^0)$. A function $u:B_{r}^+\ra \R$ is a weak solution of
\begin{align*}
L_a u &=0 \qu \trm{in } B_r^+, \\
 M_au&=h \qu \trm{on } \Ga_r^0,
\end{align*}
if $|\D u |^2 y^a \in L^1(B_r^+)$ and
\begin{align*}
\int_{B_r^+} (\D u \cdot \D \vp) y^a \dd X + \int _{\Ga_r^0} h \vp \dd x=0
\end{align*}
for all $\vp \in C^1(\ov B_r^+)$ such that $\vp \equiv 0$ on $\Ga_r^+$.
\end{definition}

The next Lemma from \cite[Theorem 6.4]{ALP} will give us the optimal regularity when $a\not =0$. (see
also \cite[Theorem 2.11]{AG}).

\begin{lemma}[Optimal regularity for $a\not=0$]\label{lem-optreg}
Let $a\not=0$ and $v\in W^{1,2}(B_1^+,y^a)$ be a bounded weak solution of
\begin{align*}
L_a v &=0 \qu \trm{in } B_1^+,\\
M_a v &= h \qu \trm{on } \Ga_1^0.
\end{align*}
If $h\in L^\infty(\Ga_{1}^0)$, then $v\in C^{1-a}(\ov B_{1/2}^+)$. Moreover, we have
\begin{align*}
\norm {v} _{C^{1-a}(\ov B_{1/2}^+)} \le C(\norm{v}_{L^\infty(B_1^+,y^a)} + \norm {h} _{L^\infty
(\Ga_{1}^0)}),
\end{align*}
where the constant $C$ depends only on $n$ and $a$.
\end{lemma}

We also recall the regularity result when $a=0$ from \cite[Section 5]{ALP}.

\begin{lemma}[Regularity for $a=0$]\label{lem-optreg0}
Let $v\in W^{1,2}(B_1^+)$ be a bounded weak solution of
\begin{align*}
\De v &=0 \qu \trm{in } B_1^+,\\
\pa_y v &= h \qu \trm{on } \Ga_1^0.
\end{align*}
If $h\in L^\infty(\Ga_{1}^0)$, then $v\in C^{\ga}(\ov B_{1/2}^+)$ for any $0<\ga<1$. Moreover, we have
\begin{align*}
\norm {v} _{C^{\ga}(\ov B_{1/2}^+)} \le C(\norm{v}_{L^\infty(B_1^+)} + \norm {h} _{L^\infty (\Ga_{1}^0)}),
\end{align*}
where the constant $C$ depends only on $n$ and $a$.\end{lemma}

The following Liouville type theorem from \cite[Lemma 2.7]{CSS} will be used in \Cref{sec:blowups}.

\begin{lemma}[Liouville type theorem]\label{lem-liouville}
Let $v$ be a harmonic function in $\R^{n+1}$ such that $v(x,y)=v(x,-y)$ for all $x\in \R^n$ and $y\in
\R$. If $v$ has a polynomial growth, i.e.
\begin{align*}
|v(X)|\le C(1+|X|^k)
\end{align*} for some constant $C$ and degree $k$, then $v$ is a polynomial of degree at most $k$.
\end{lemma}

\section{Physical Interpretation}\label{sec:phy}
In this section we shall show the physical interpretation of problem  \ref{pb:N} in terms of the basic
fractional frequency. Fractional frequency is better understood when $s=\2$, since it is related to the
classical Steklov eigenvalue problem.

The local case ($s=1$) for the composite membrane problem has been well studied (see \cite{CGI+} and
related references). Our non-local optimization question is linked to the following:

\begin{named}{Problem $\mathbf{(P_N)}$}\label{pb:P1}
Build a body of prescribed shape out of given materials of varying density, in such a way that the body
has prescribed mass and so that the basic fractional frequency (with fixed boundary) is as small as
possible.
\end{named}

To give the exact mathematical formulation of this problem, we define the class of admissible densities
by
\begin{align*}
\mathcal{P}= \left\{ \rho : \Om \ra [h,H] : \int_\Om \rho(x) \dd x= M \right\},
\end{align*}
where $h$, $H$, and $M$ are the given constants satisfying $0\le h < H$, $0<M\in [h|\Om|, H |\Om|]$.
Then \ref{pb:P1} is to find a density $\rho$ and a body $u$ which achieve the double infimum in
\begin{align}\label{PN-inf}
\Theta (h,H,M) := \inf_{\rho \in \cP} \inf_{u\in \cH\sm \{0\}} \fr{\fr{c_{n,s}}{2}\int_{\R^n\ti
\R^n}\fr{|u(x)-u(y)|^2}{|x-y|^{n+2s}}\dd x \dd y}{\int_\Om \rho(x) u^2(x)\dd x}.
\end{align}
The associated Euler-Lagrange equation is
\begin{equation}\label{EL-PN}
\begin{split}
(-\De)^s u&= \Theta \rho u \qu \trm{in } \Om,\\
u&= 0 \qu \trm{in } \R^n\sm \Om.
\end{split}
\end{equation}
Moreover, $u$ has a sign in $\Om$ if it is not a constant function.
\begin{lemma}\label{lem-sign}
Let $u$ be a function achieving the infimum in \eqref{PN-inf}. Then $u$ has a sign, i.e., either $u>0$
in $\Om$ or $u<0$ in $\Om$ holds.
\end{lemma}

\begin{proof}
Since \eqref{PN-inf} is invariant under the constant multiplication, we may assume that $\int_\Om
\rho(x) u^2(x)\dd x =1$.
Now we consider the positive part $u_+=\max\{u,0\}$ and the negative part $u_-=\max\{-u,0\}$, and define
\begin{align*}
J_+= \int_\Om \rho(x) u_+^2(x)\dd x \qu \trm{and} \qu J_-= \int_\Om \rho(x) u_-^2(x)\dd x
\end{align*}
so that $J_++J_-=1$. Observe that
\begin{equation}
\begin{split}\label{ineq-sign}
|u(x)-u(y)|^2 
&\ge |u_+(x)-u_+(y)|^2+|u_-(x)-u_-(y)|^2,
\end{split}
\end{equation}
which yields
\begin{align*}
\Theta(h,H,M)&\geq\fr{c_{n,s}}{2}\int_{\R^n\ti
\R^n}\fr{|u_+(x)-u_+(y)|^2}{|x-y|^{n+2s}}+\fr{c_{n,s}}{2}\int_{\R^n\ti
\R^n}\fr{|u_-(x)-u_-(y)|^2}{|x-y|^{n+2s}}\\
&\ge \Theta(h,H,M) J_++\Theta(h,H,M) J_- = \Theta(h,H,M).
\end{align*}
Therefore the inequality in \eqref{ineq-sign} should be an equality, that is, either $u_- \equiv 0$ or
$u_+\equiv 0$ holds. To finish the proof let us assume, without loss of generality, that the first case
occurs  so that $u\ge0$ in $\Om$. By applying the strong maximum principle (e.g. \cite[Theorem
2.3.3]{BV}) to \eqref{EL-PN}, we obtain $u>0$ in $\Om$. For the second case, we consider $-u$ instead of
$u$.
\end{proof}

To see that \ref{pb:P1} is contained in \ref{pb:N}, we need the following density representation Lemma, which is essentially the ``bathtub principle" (see Section 1.14 in \cite{LL}):
\begin{lemma}\label{density representation}
Assume that $(u,\rho)$ is a minimizer for \ref{pb:P1}. For a set $D$ such that $\{u <t \} \su D \su
\{u\le t\} $ where $ t:=\sup\{c:|\{u<c\}|<A\}$, let us define $\rho_D= h\chi_D +H \chi_{D^c}$. Then
$(u,\rho_D)$ is also a minimizer for \ref{pb:P1}.
\end{lemma}

\begin{proof}
From \Cref{lem-sign}, $u$ has a sign so we may assume $u>0$ in $\Om$. Now the conclusion follows from
\eqref{PN-inf} and the inequality
\begin{align*}
\begin{split}
\int_{\Om}(\rho_D-\rho) u^2=&\left(\int_{\{u<t\}} +\int_{\{u=t\}}+ \int_{\{u>t\}}  \right)(\rho_D-\rho)
u^2\\
\ge& \left(\int_{\{u<t\}} +\int_{\{u=t\}}+ \int_{\{u>t\}}  \right)(\rho_D-\rho) t^2\\
=&0.
\end{split}
\end{align*}
\end{proof}

Unless otherwise stated, a density function for the minimizer is always of the form given by
\Cref{density representation}. Notice also that $\rho_D$ is unique up to a measure zero set if and only
if $|\{u=t\}|=0$.

In the following lemma we see the relation between \ref{pb:P1} and \ref{pb:N}. Since the proof is
identical to \cite[Theorem 13]{CGI+}, we omit it here.

\begin{lemma}\label{PN relation}
\ref{pb:P1} is solved by a pair $(u,\rho_D)$ achieving $\Theta(h,H,M)$  if and only if \ref{pb:N} is
solved by a pair $(u,D)$ achieving $\La (\al ,A)$, where the parameters and the minimal eigenvalue are
related by
\begin{align*}
\al &= (H-h)\Theta(h,H,M),\\
A &= \fr{H|\Om|-M}{H-h},\\
\La(\al,A)&= H\Theta(h,H,M).
\end{align*}
Moreover, for  any   $0\le h <H$, the possible value of the parameters  is precisely $0<\al \le
\ov\al_\Om(A)$ if $0\le A <|\Om|$, where the constant $\ov \al _\Om(A)$ will be defined in
\eqref{const-alpha}, and $0<\al <\infty$ if $A=|\Om|$.
\end{lemma}

Until now, we saw that our nonlocal version of the composite membrane problem, \ref{pb:N}, is a
generalization of the physical problem, \ref{pb:P1}. Our next task is to explain the meaning of
fractional frequency for $s=\2$ by considering the optimization problem for the  Steklov eigenvalue
under the presence of a density $\rho$ (this is ``weighted").

Let $0<h<H$ and $0< M\in [h|\Om|,H|\Om|]$. As in  \ref{pb:P1}, we define the class of admissible
densities by
\begin{align*}
\cP_{S} = \left\{ \rho : \pa \Om \ra [h,H] : \int_{\pa \Om} \rho(x) \dd\si(x) = M \right\},
\end{align*}
for a $C^2$-boundary $\pa \Om$ of a bounded domain $\Om \su \R^{n+1}$, and for each $\rho \in \cP$, the
class of admissible functions by
\begin{align*}
\cH_{S}[\rho] = \left\{ u\in H^1(\Om) : \int_{\pa\Om} \rho u \dd\si=0 \right\}.
\end{align*}

Now we state the optimization problem for the (weighted) Steklov eigenvalue:
\begin{named}{Problem $\mathbf{(S)}$}\label{pb:S}
Find a density $\rho \in \cP_S$ and a function $u \in \cH_S[\rho]$ which realize the double infimum in
\begin{align*}
\Theta_S(h,H,M) := \inf_{\rho\in\cP_S} \inf_{u\in \cH_S[\rho]\sm\{0\}} \fr{\int_\Om |\D u|^2 \dd
x}{\int_{\pa\Om} \rho u^2\dd\si},
\end{align*}
whose Euler-Lagrange equation is
\begin{equation*}
\begin{alignedat}{3}
\De u &=0 &&\qu \trm{in } \Om,\\
\fr{\pa u}{\pa \nu} &= \Theta_S \rho u &&\qu \trm{on } \pa \Om.
\end{alignedat}
\end{equation*}
\end{named}

According to \cite{LP}, for a given density $\rho$ and $\Om \su \R^2$, the physical meaning of the
Steklov problem is constructing a free elastic membrane with prescribed mass concentrated at the
boundary (see also \cite{GP}.) Moreover, from the well known identification $(-\De)^\2 = \fr{\pa}{\pa
\nu}$ when $\Om=\R^{n+1}_+$ (cf. \cite{CS07}),
\ref{pb:P1} can be thought as the mixed boundary version of \ref{pb:S}, i.e.,
\begin{equation*}
\begin{alignedat}{3}
\De u& = 0 &&\qu \trm{in } \Om,\\
u&=  0 &&\qu\trm{on } A_0,\\
\fr{\pa u}{\pa \nu} &=  \Theta_S \rho u &&\qu \trm{on } A_1,
\end{alignedat}
\end{equation*}
where $A_1$ is a domain in $\pa\Om$, and $A_0=\pa\Om\sm A_1$ (see \cite{BKPS} for mixed Steklov
problems). Keeping in mind this interpretation, we may think that the basic fractional frequency in
\ref{pb:P1} describes the basic frequency of a elastic membrane which is fixed on $A_0$ and free on
$A_1$, with prescribed mass concentrated on $A_1$.

Notice that we can generalize \ref{pb:S} in the same way as \ref{pb:P1} does. In fact, the analogies of
\Cref{density representation} and \Cref{PN relation} can be established with minor modifications, for
example, the set $D$ should be of the form
\begin{align*}
\{-t < u <t\} \su D \su \{ -t \le u \le t\},
\end{align*}
since $u$ is no longer a positive function in the Steklov problem.

\section{Basic properties}\label{sec:basic}
In this section we establish existence for \ref{pb:N} and investigate the parameter dependence on $\La$.

\begin{lemma}\label{lem-regularity}
For any $\al>0$ and $A\in [0,|\Om|]$ there exists an optimal pair $(u,D)$ satisfying
\begin{align*}
u\in H^{2s}_{loc}(\Om) \cap C^{\beta} (\Om) \cap C^s(\R^n),
\end{align*}
where $\beta$ is any constant in $(0,2s)$.
\end{lemma}

\begin{proof}
As in \cite{CGI+}, we first investigate a regularity of a weak solution to
\begin{align*}
(-\De)^s u +\rho u =0 \quad \trm{in } \Om,
\end{align*}
where $\rho$ is a bounded function.

By the result of 	\cite{BWZ}, $u$ belongs to $H^{2s}_{loc}(\Om)$. From Lemma 2.3 in \cite{MPSY}, we
can use a standard bootstrapping argument so that $u\in L^{\infty}(\Om)$. Then by Proposition 1.1 in
\cite{RS} gives that $u\in C^s(\R^n)$. Moreover, from Lemma 2.9 in the same paper, we have $u\in
C^{\beta}(\Om)$ for any $\beta\in (0,2s)$.

Now we take a minimizing sequence $\{(u_j,D_j)\}$ such that $u_j$ is a positive $L^2$-normalized first
eigenfunction of $(-\De)^s+\al\chi_{D_j}$ in $H^s_0(\Om)$, i.e., $u_j$ minimizes the functional 
\begin{align*}
\fr{c_{n,s}}{2}\left(\int_\Om\int_\Om \fr{|w(x)-w(y)|^2}{|x-y|^{n+2s}}\dd x \dd y\right)^2+\al \int_{D_j}w^2 \dd x
\end{align*} 
among all functions $w\in H^s_0(\Om)$ that satisfy $\norm{w}_{L^2(\Om)}=1$. The existence of such eigenfunctions is proved in \cite{DFL18_MN}.
Since $\la(D_j)$ is bounded, $u_j$ is also
bounded in $H^s_0(\Om)$. Note that $\chi_{D_j}$ is bounded in $L^2(\Om)$. Then, up to subsequence, we
have
\begin{align*}
\chi_{D_j} &\rightharpoonup\eta \quad \trm{in } L^2(\Om)\\
u_j &\rightharpoonup u \quad \trm{in } H^s_0(\Om).
\end{align*}
By compactness (see \cite{NPV}), we can find a strongly convergent subsequence $\{u_j\}$ in $L^2(\Om)$.
Thus we know that $\chi_{D_j}u_j \rightharpoonup \eta u$ in $L^{2}(\Om)$. Therefore, $u$ is a weak
solution of
\begin{align*}
(-\De)^s u+\al \eta u =\La u.
\end{align*}
From the properties of weak convergence, we have
\begin{align*}
0\le \eta \le 1, \quad \int _\Om \eta =A.
\end{align*}
This, and the following inequality
\begin{align}
\begin{split}\label{bathtub}
\int_{\Om}(\eta-\chi_D) u^2&=\left(\int_{\{u<t\}} +\int_{\{u=t\}}+ \int_{\{u>t\}}  \right)(\eta-\chi_D)
u^2\\
&\ge \left(\int_{\{u<t\}} +\int_{\{u=t\}}+ \int_{\{u>t\}}  \right)(\eta-\chi_D) t^2\\
&=0,
\end{split}
\end{align}
where $ t:=\sup\{c:|\{u<c\}|<A\}$ and for any $D$ satisfying
\begin{align}\label {eq-D1}
\{u<t\} \subset D\subset \{u\le t\}, \qu |D|=A,
\end{align}
imply that we may replace $\eta$ by $\chi_D$.
\end{proof}

Note that Lemma \ref{lem-sign} implies that $u$ has a sign. With loss of generality, assume in the following that $u>0$. Now we remark that, from inequality \eqref{bathtub}, the optimal configuration $D$ always has the form
\eqref{eq-D1}, and thus $D$ contains a tubular neighborhood of $\pa \Om$.

\begin{lemma}\label{lem-D}
Let $(u,D)$ be an optimal pair with $A>0$. Then:
\begin{enumerate}[(i)]
\item The optimal configuration satisfies
\begin{align}\label{eq-D2}
\{u<t\} \su D\su \{u\le t\},
\end{align}
where $t:=\sup \{c:|\{u<c\}|<A\}>0$.
\item $D$ contains a tubular neighborhood of the boundary $\pa \Om$.
\end{enumerate}
\end{lemma}

As indicated in the introduction, one of the main results in this paper is to show that  $D$ is exactly
a sublevel set (\Cref{lem-D-main}), i.e.,
\begin{align}\label{D-sublevel}
D=\{x\in \Om:u(x)\le t\}.
\end{align}
The proof is very delicate and uses the blowup arguments in \Cref{sec:blowups}.

We close this section with the parameter dependence on $\La$. The proof is standard with some necessary
minor variations from \cite[Proposition 10]{CGI+}, but we include it here for convenience of the reader.

\begin{lemma}
The optimal frequency $\La(\al, A)$ is strictly increasing in each variable on $\R_+^2$, and
$\La(\al,A)-\al$ is strictly decreasing in $\al$ for fixed $A>0$. Furthermore, the function
$(\al,A)\mapsto \La(\al,A)$ is Lipschitz continuous, uniformly on bounded sets, i.e., for any
$\al_1,\al_2\ge0$, and $A_1,A_2\in [0,|\Om|]$,
\begin{align}\label{ineq-lem-La}
|\La(\al_1,A_1)-\La(\al_2,A_2)|\le \fr{\max\(A_1,A_2\)}{|\Om|}|\al_1-\al_2|+C|A_1-A_2|,
\end{align}	
where $C=C(\Om, \max\{\al_1,\al_2\})$. Consequently, there exists a unique value $\ov \al_\Om(A)$ for
$A\in [0,|\Om|)$ satisfying
\begin{align}\label{const-alpha}
\La (\ov \al_\Om(A),A)= \ov \al _\Om (A),
\end{align}
and the function $A \mapsto \ov\al_\Om(A)$ is continuous and strictly increasing with
$\ov\al_\Om(0)=\mu_\Om$ and $\ov \al_\Om(A)\ra \infty$ as $A \ra |\Om|$.
\end{lemma}

\begin{proof}
Let $\La_i=\La(\al_i,A_i)$ and let $(u_i,D_i)$ be a minimizer for $\La_i$ such that $\int_\Om u_i^2=1$
for $i=1,2$. Then, we have
\begin{align*}
\La_i= \int_{\R^n} |(-\De)^{s/2} u_i|^2 +\al_i \int_{D_i} u_i^2, \quad |D_i|=A_i.
\end{align*}
We may assume $A_1\le A_2$, and then take $D_1' \su D_2$ with $|D_1'|=A_1$ and $D_2' \supset D_1$ with
$|D_2'|=A_2$. From the optimality one obtains
\begin{align}\label{ineq-La}
\La_i \le \int_{\R^n} |(-\De)^{s/2} u_j|^2 +\al_i \int_{D_j'} u_j^2 = \La_j -\al_j\int_{D_j}
u_j^2+\al_i\int_{D_i'}u_j^2
\end{align}
or, equivalently,
\begin{align}\label{ineq-La-al}
\La_i-\al_i\le \La_j-\al_j+\al_j\int_{\Om\setminus D_j}u_j^2-\al_i\int_{\Om\sm D_i'}u_j^2
\end{align}
for all $i,j\in\{1,2\}$. Then, \eqref{ineq-La} with $(i,j)=(1,2)$ and $\al_1=\al_2$ yields
\begin{align*}
\La_2-\La_1\ge\al_1\int_{D_2\sm D_1'}u_2^2\ge 0.
\end{align*}
Moreover, equality in the above holds if and only if $u_2\equiv 0$ on $D_2\sm D_1'$ or $\al_1=0$. By the
global strong maximum principle, the former case cannot happen unless $A_1=A_2$. In fact, for $A_2>A_1$,
$|D_2\sm D_1'|>0$ so that $u_2\equiv0$ on $D_2\sm D_1'$, which cannot happen. This proves $\La(\al,A)$
is strictly increasing in $A$.

Similarly, \eqref{ineq-La} with $(i,j)=(1,2)$ and $A_1=A_2$ gives
\begin{align}\label{ineq-al}
\La_2-\La_1\ge (\al_2-\al_1)\int_{D_2}u_2^2>0
\end{align}
for $\al_2>\al_1$, and \eqref{ineq-La-al} with $(i,j)=(2,1)$ and $\al_1=\al_2$ implies
\begin{align*}
(\La_1-\al_1)-(\La_2-\al_2)\ge \al_1\int_{D_2'\sm D_1} u_1^2>0
\end{align*}
for $A_2>A_1$.

For the second part, combining \eqref{ineq-La} with $(i,j)=(1,2)$ and $(i,j)=(2,1)$, we obtain
\begin{align}\label{ineq-bothside}
(\al_2-\al_1)\int_{D_2}u_2^2+\al_1\int_{D_2\sm D_1'}u_2^2\le \La_2-\La_1\le
(\al_2-\al_1)\int_{D_2'}u_1^2+\al_1\int_{D_2'\sm D_1}u_1^2,
\end{align}
so that
\begin{align}\label{ineq-Lip}
|\La_2-\La_1|\le |\al_2-\al_1|\max\(\int_{D_2'}u_1^2,\int_{D_2}u_2^2\)+\al_1\max\(\int_{D_2'\sm
D_1}u_1^2,\int_{D_2\sm D_1'}u_2^2\).
\end{align}
From \Cref{lem-D}, $D_2$ satisfies \eqref{eq-D2} with $u_2$ and $t_2:=\sup\{s:|\{u_2<s\}<A_2\}$.
Moreover, we may take $D_2'$ of the form \eqref{eq-D2} for $u_1$ and some $t_2'$ since $(u_1,D_1)$ is
also optimal pair. Now observe that for any $D\su \Om$ satisfying \eqref{eq-D2} with any $s>0$ we have
\begin{align*}
\fr{\int_{D}u^2}{|D|}\le\fr{\int_\Om u^2}{|\Om|},
\end{align*}
which comes from the fact that average on the whole space is greater than the average on the set $\{u\le
t\}$.
Then
\begin{align}\label{ineq-L2}
\max\(\int_{D_2'}u_1^2,\int_{D_2}u_2^2\)\le \fr{A_2}{|\Om|}\le 1.
\end{align}
On the other hand, using \eqref{ineq-bothside} with $\al_2=0$, $A_1=A_2$, one has $\La_i\le
\mu_\Om+\al_i$ for $i=1,2$. Moreover, $u_i$ solves
\begin{align*}
(-\De)^{s}u_i+(\al\chi_{D_i}-\La_i)u_i=&0 \quad \trm{in }\Om,\\
u=&0 \quad \trm{on } \R^n\sm \Om,
\end{align*}
whose coefficients are bounded if $\al$ is bounded. By global boundedness (c.f. Lemma 2.3 in
\cite{BWZ}), we obtain
\begin{align*}
\max\(\int_{D_2'\sm D_1}u_1^2,\int_{D_2\sm D_1'}u_2^2\)\le |A_2-A_1| \max \(\sup_\Om u_1^2,\sup_\Om
u_2^2\)\le C|A_2-A_1|,
\end{align*}
where $C=C(\Om,\al_1,\al_2)$, and hence \eqref{ineq-lem-La} follows from \eqref{ineq-Lip},
\eqref{ineq-L2} and the estimate above.

Finally, we observe that $\La(\al,A)-\al$ equals $\mu_\Om>0$ for $\al=0$, and goes to $-\infty$ as $\al
\ra \infty$ since $\La(\al,A)-\al =\La-\mu_\Om+\mu_\Om-\al\le -(1-\fr{A}{|\Om|})\al+\mu_\Om$ by taking
$A_1=A_2=A$ and $\al_2=0$ in \eqref{ineq-lem-La}. Therefore, the function $\ov \al_\Om$ is well-defined.
Again, \eqref{ineq-lem-La} implies the continuity assertion and the first inequality of
\eqref{ineq-bothside} gives monotone assertion if we choose $\al_1=\ov\al_\Om(A_1)$ and
$\al_2=\ov\al_\Om(A_2)$. Then the remaining assertions are $\ov\al_\Om(0)=0$, which is trivial, and
$\ov\al_\Om(A)\ra \infty $ as $A\ra |\Om|$. This follows at once by observing
\begin{align*}
\ov\al_\Om\int_{\Om\sm D_2}u_2^2 \ge \mu_\Om>0
\end{align*}
from \eqref{ineq-al} with $\al_2=\ov \al_\Om$ and $\al_1=0$.
\end{proof}

\section{Blowups}\label{sec:blowups}

Fix $-1<a =1-2s<1$. In this section and the next one we will consider blowups for  \ref{pb:E} and its
non-triviality on $\R^n\ti \{0\}$. The results obtained will be used in \Cref{sec:struc} to show that an
optimal configuration $D$ is given by the sublevel set of the corresponding solution $u$, i.e.,
$D=\{u\le t\}$ for some $t$, where $(u,D)$ is an optimal pair. We first discuss the case of $a\not=0$
($s \not = \2$), for which the optimal regularity is given by \Cref{lem-optreg}, and then the remaining
case, $a=0$ (this is, $s=\2$), will be treated.

Throughout this section, \ref{pb:E} is converted into a more general problem as in \cite{CK} by defining
$v=t-u$, $f=(\La-\al )u$, and $g=-\La u$, namely
\begin{equation}
\begin{alignedat}{3}\label{eq-v}
	L_a v = \mathrm{div}(y^a \D v)&=0 &&\quad \textrm{in } \R^{n+1}_+,\\
	M_a v = \lim_{y\ra 0}(y^a \pa _y v)&= f\chi_{D}-g\chi_{D^c} &&\qu\textrm{on } \Om\subset \pa
\R^{n+1}_+,\\
	v&=t \quad &&\qu\textrm{on } \pa\R_+^{n+1}\setminus \Om,
\end{alignedat}
\end{equation}
where $t$ is given by \eqref{eq-D2}. Notice that $f>0$, $g<0$, $f+g<0$ in a neighborhood of the free
boundary $\mathcal{F}$, and $\{ v>0\} \subset D \subset \{ v\ge 0\}$. Since most of the properties in
this section use a local argument near the free boundary, we focus on a half-ball $B^+_{r_0}$ centered
on the free boundary $\cF$ with small radius $r_0>0$ so that $\Ga_{r_0}^0\Subset \Om$. By translation,
we may assume that the center of this half-ball is the origin. We also assume that for some positive
constant $\eta_0$,

\begin{align}\label{const-lambda}
f \ge \eta_0 >0,\qu g\le -\eta_0 <0,\qu  \trm{and}\qu f+g\le -\eta_0 <0,
\end{align}
over a ball $\Ga_{r_0}^0$, and that $f, g \in C^s(\ov\Ga_{r_0}^0)$. In the rest of this section $v$ will
always denote a weak solution of
\begin{align}\label{eq-v-ball}
L_a v=0 \qu \trm{in } B_{r_0}^+ \qu \trm{and} \qu  M_a v = \fc \qu  \trm{on } \Ga_{r_0}^0
\end{align}
in the  sense of Definition \ref{def-weak-local}.

From the standard argument of Caccioppoli's inequality (see for instance \cite{HL}) we obtain the
following energy estimate:

\begin{lemma}[Energy estimate]\label{lem-EE}
Let $-1<a<1$ and $v$ be a weak solution of \eqref{eq-v-ball}. Then for any $0<r<r_0$ we have
\begin{align}\label{ineq-energy}
\int_{B_{r/2}^+} |\D v|^2 y^a \dd X \le \fr{32}{r^2}\int_{B_{r}^+} v^2 y^a \dd X+
2\max\{\norm{f}_{L^\infty},\norm{g}_{L^\infty}\}\int_{\Ga_{r}^0} |v| \dd x.
\end{align}
\end{lemma}

Let us now define the scaled function
\begin{align*}
v_r(X)=\fr{v(rX)}{r^{1-a}},\qu f_r(x)=f(rx),\qu \trm{and}\qu g_r(x)=g(rx),
\end{align*}
and the scaled configuration set $D_r=\{x\in \R^n:rx\in D\}$
for $0<r<r_0$. Notice that we assumed $0\in \pa D$. If $a<0$, we also assume that $Dv(0)=0$. Observe that
\begin{equation*}\begin{aligned}[3]\label{eq-vr}
L_a v_r=\dv(y^a\D v_r )&=0 \qu &&\trm{in } B_{\fr{1}{r}}^+,\\
M_a v_r=\lim_{y\ra 0}y^a\pa _y v_r&=f_r\chi_{D_r}-g_r \chi_{D_r^c} \qu &&\trm{on } \Ga^0_\fr{1}{r}.
\end{aligned}\end{equation*}
In terms of the scaled function above, inequality \eqref{ineq-energy} becomes
\begin{align}\label{ineq-energy-r}
\int_{B_{1/2}^+} |\D v_r|^2 y^a \dd X \le 32\int_{B_1^+} v_r^2 y^a \dd X+ 2
\max\{\norm{f_r}_{L^\infty},\norm{g_r}_{L^\infty}\}\int_{\Ga_1^0} |v_r|\dd x.
\end{align}
This, together with \Cref{lem-optreg}, yields:

\begin{lemma}\label{lem-blowup-not0}
Let $a\not = 0$. Assume that $v$ is a bounded weak solution of \eqref{eq-v-ball}. Then there exist a decreasing
subsequence $\{r_j\}$ converging to zero and a function $v_0$ in $C^{1-a}_{loc}(\R^{n+1}_+)$ such that,
for any $R>0$, $v_{r_j}\ra v_0$ in $C^{\ga}(\ov B_R^+)$ for $\ga<1-a$, and $v_{r_j}\rightharpoonup v_0$
in weakly in $H^1(B_R^+,y^a)$ as $j\ra \infty$. Moreover, the function $v_0$ weakly solves the equation
\begin{equation*}
\begin{alignedat}{3}
L_a v_0 &=0 \qu &&\trm{in } \R^{n+1}_+,\\
M_a v_0 &=h_0 \qu &&\trm{on } \R^n,
\end{alignedat}
\end{equation*}
for a function $h_0\in L^2_{loc}(\R^n)$ satisfying $h_0\ge \la>0$, where $\lambda$ is the constant in
\eqref{const-lambda}.
\end{lemma}

\begin{proof}
The convergence in $C^\ga(\ov{B}_R^+)$ and the limit $v_0\in C^{1-a}_{loc}(\R_+^{n+1})$ follow from \Cref{lem-optreg} since $a\not=0$, and the weak convergence is a consequence of the estimate \eqref{ineq-energy-r}.
Thus it suffices to show the second part. Let
$h_r:=f_r\chi_{D_r} -g_r \chi _{D_r^c}$. Observe that $h_r \in L^2_{loc}(\R^n)$ and $h_r \ge \la >0$.
Then there exists a weakly convergent subsequence $\{h_{r_j}\}_{j\in \mathbb{N}}$ and a function $h_0\in
L^2_{loc}(\R^n)$ such that $h_{r_j} \rightharpoonup h_0$ and therefore $h_0\ge \la >0$. \end{proof}

The function $v_0$ is called a \emph{blowup} of $v$ at the origin. For a later use, we also introduce
blowups over a sequence. Let $\{x_j\}$ be a convergent sequence whose limit is $x_0$. We consider the
limits $v_{r_j,x_j}\ra v_0$ in $C^{1-a}_{loc}(\R^{n+1}_+)$ as $j\ra \infty$, where
$v_{r_j,x_j}(x)=v(x_j+r_jx)/r^{1-a}$.
We call such $v_0$ a blowup over the sequence $x_j\ra x_0$.\\

In the case of $a=0$, the scaled functions $\{v_r\}$ may not be uniformly bounded in $L^\infty(B_1^+)$
so that we consider slightly different functions. For this, it will be useful to define the following
quantity:
\begin{align*}
C_r := \sup _{B_1^+} \fr{v(rX)}{r}.
\end{align*}

\begin{lemma}\label{lem-class}
Let $a=0$, $R>0$, and $v$ be a bounded weak solution of \eqref{eq-v-ball}.
\begin{enumerate}[(i)]
\item If $\displaystyle\sup_{0<r<r_0} C_r<\infty$, then there exist a decreasing subsequence $\{r_j\}$
    converging to zero and a function $v_0$ in $C^{0,1}_{loc}(\R^{n+1}_+)$ such that for any $R>0$,
    $v_{r_j}\ra v_0$ in $C^{\ga}(\ov B_R^+)$ for $\ga<1$ and $v_{r_j} \rightharpoonup v_0$ in weakly
    in $H^1(B_R^+)$ as $j\ra \infty$. Moreover, the function $v_0$ weakly solves the equation
\begin{equation*}
\begin{alignedat}{3}
\De v_0 &=0 \qu &&\trm{in } \R^{n+1}_+,\\
\pa_y v_0 &=h \qu &&\trm{on } \R^n,
\end{alignedat}
\end{equation*}
with the function $h\in L^2_{loc}(\R^n)$ satisfying $h\ge \la >0$, where $\la$ is the constant in
\eqref{const-lambda}.
\item If $\displaystyle \sup_{0<r<r_0} C_r = \infty$, then there exist a decreasing subsequence
    $\{r_j\} $ converging to zero and a function $\tilde v_0$ such that for any $R>0$, $\tilde
    v_{r_j}:=v_{r_j}/C_{r_j} \ra \tilde v_0$ in $C^\ga (\ov B_R^+)$ for $\ga <1 $ and $\tilde
    v_{r_j}\rightharpoonup \tilde v_0$ in weakly in $H^1(B_R^+)$ as $j\ra \infty$. Moreover, there is
    a nonzero vector $(a_1,\cdots,a_n)\in \R^n$ so that
\begin{align}\label{eq-v0}
\tilde v_0(X) = (a_1,\cdots,a_n) \cdot x, \qu \trm{for } X=(x,y)\in \R^n.
\end{align}
\end{enumerate}
\end{lemma}

\begin{proof}
The first case follows as in the argument in \Cref{lem-blowup-not0}. For the second case, we can choose
a subsequence $\{r_j\}$ such that
$C_{r_j} \ra \infty$ as $j\ra \infty$ and
\begin{align*}
C_{r_j} \ge \sup_{r_j\le r\le r_0} C_r.
\end{align*}
Then for any $0<R < r_0/r_j$, we have
\begin{align*}
\int_{B_R^+} \D \tilde v_{r_j}\cdot \D \vp \dd X \le
\fr{\max\{\norm{f}_{L^\infty(\Ga_R^0)},\norm{g}_{L^\infty(\Ga_R^0)}\}}{C_{r_j}}\int_{\Ga_R^0}|\vp|\dd x
\end{align*}
for all $\vp \in C^1(\ov B_R^+)$ such that $\vp \equiv 0$ on $\Ga_R^+$. Moreover, we observe that
$\tilde v_{r_j}(0)=0$, $\sup_{B_1^+} \tilde v_{r_j} =1$, and for $R\ge 1$,
\begin{align*}
\sup_{B_R^+} \tilde v_{r_j} =
\fr{\sup_{B_R^+}v_{r_j}}{C_{r_j}}=\fr{\sup_{B_1^+}v_{r_jR}}{C_{r_j}}R=\fr{C_{r_jR}}{C_{r_j}}R\le R.
\end{align*}
From \Cref{lem-optreg0} and \Cref{lem-EE}, there exist a subsequence, again denoted by $\{r_j\}$, and a
function $\tilde v_0$ in $C^{0,1}_{loc}(\ov \R^{n+1}_+)$ such that for any $R>0$, $v_{r_j} \ra v_0$ in
$C^\ga(\ov B_R^+)$ for $\ga<1$ and $\tilde v_{r_j} \rightharpoonup \tilde v_0$ in weakly in $H^1(B_R^+)$
as $j\ra \infty$. Thus we have $\tilde v_0(0)=0$, $\sup_{B_1^+} \tilde v_0 =1$, $\sup_{B_R^+}\tilde
v_0\le R$, and
\begin{align*}
\int_{B_R^+} \D \tilde v_0 \cdot \D\vp \dd X = 0
\end{align*}
for all $\vp\in C^1(\ov B_R^+)$ such that $\vp \equiv 0$ on $\Ga_R^+$. Notice that the last equality
follows by considering $-\vp$ instead of $\vp$. If we evenly reflect $\tilde v_0$ across $\{y=0\}$, then
the new function, still denoted by $\tilde v_0$,  is harmonic in $\R^{n+1}$ satisfying
\begin{align*}
\tilde v_0(X) \le 1+|X| \qu \trm{for all } X\in \R^{n+1}.
\end{align*}
By the Liouville type result from \Cref{lem-liouville}, $\tilde v_0$ is a polynomial of degree at most
one. Now \eqref{eq-v0} can be deduced from the fact $\tilde v_0(0)=0$, $\sup_{B_1} \tilde v_0=1$, and
$\tilde v_0$ is even in $y$-variable.
\end{proof}

\section{Nondegeneracy}\label{sec:nondeg}
In this section we will show nondegeneracy, which will imply that blowups are not identically zero over
$\R^n$. Notice that from the optimal regularity near $\pa D$, we see that $|t-u(x)| \le C \dist (x,\pa D)^{2s}$ for some constant $C$. Nondegeneracy gives the opposite inequality.
More precisely:
\begin{lemma}\label{lem-nondeg-pre}
Let $(u,D)$ be an optimal pair and $a\not=0$. It holds:
\begin{enumerate}[(i)]
\item There exist positive constants $c_0$ and $C_0$ such that if $x\in \{u<t\}$ and $\dist(x,\pa D)
    \le c_0$, then
\begin{align}\label{nondeg-1}
u(x)\le t- C_0\dist (x,\pa D)^{2s}.
\end{align}
\item There exist positive constants $c_0$ and $C_0$ such that if $x\in \{u>t\}$ and $\dist(x,\pa D)
    \le c_0$, then
\begin{align}\label{nondeg-2}
u(x)\ge t+ C_0\dist (x,\pa D)^{2s}.
\end{align}
\end{enumerate}
\end{lemma}

\begin{proof}
Fix a point $x_0$ in $\{u<t\}$, and let $d_0=\dist(x_0,\pa D)>0$ and $\beta=t-u(x_0)>0$. We may assume $x_0$ is the origin. Denote by $u_1$  the extension of $u$ to $\R^{n+1}_+$ through \eqref{EL-E}. We set
$w:=t-u_1(x,y)+t(\La-\al)(1-a)^{-1}y^{1-a}$ with $1-a=2s$, which satisfies
$$L_a w=0\quad\text{and} \quad-M_aw=(\La-\al)w.$$
Applying Harnack's inequality (see
\cite{TX}) to   $w$ in a neighborhood of $x_0$, we have
\begin{align*}
\underline{c} \beta \le t-u(x) \le \overline{c} \beta \qu \trm{in } B_{d_0/2}
\end{align*}
for some positive constants $\underline{c}$ and $\overline{c}$.
Now we define
\begin{align*}
\tilde u(x)=\begin{cases}
\max\{ u(x),t-\ov{c}\beta \psi(x)\} \qu &\trm{if }x\in B_{d_0/2},\\
u(x) \qu &\trm{otherwise},
\end{cases}
\end{align*}
where
$\psi$ is a radial cut-off function such that $\psi\equiv0$ in $B_{d_0/4}$ and $\psi \equiv 1$ outside
$B_{d_0/2}$.

We are going to use the following inequality: given $A',A'',B',B''>0$, if $B'/A' \le B''/A''$, then
$(A''-A')B'/A' \le B''-B'$. From this inequality, together with the minimality of $\La$, we have
\begin{align*}
\La\int_\Om \tilde{u}^2\dd x - \La\int_\Om u^2\dd x \le \|(-\De)^{s/2}\tilde u\|^2-\|(-\De)^{s/2}
u\|^2 +\al \int_D \tilde{u}^2\dd x - \al \int _D u^2\dd x.
\end{align*}
Since $\tilde u\ge u\ge0$ and $D\su \Om$, we arrive at
\begin{align}\label{ineq-minimality}
(\La-\al)\left(\int_\Om \tilde{u}^2\dd x - \int_\Om u^2\dd x\right) \le \|(-\De)^{s/2}\tilde
u\|^2-\|(-\De)^{s/2} u\|^2.
\end{align}
To  further proceed, we observe that
\begin{align*}
\begin{split}
\int_\Om \tilde{u}^2\dd x - \int_\Om u^2 \dd x&\ge \int_{B_{d_0/4}}t^2-(t-\underline{c}\beta)^2
=|B_{d_0/4}|(2t \underline{c}\beta-\underline{c}^2 \beta^2),
\end{split}
\end{align*}
and that for $K:=\{x\in B_{d_0/2}: u(x)<t-\ov{c} \beta \psi(x)\}$ the following inequalities hold:
\begin{enumerate}[(i)]
\item if $x\in K$ and $y \in K^c \cap B_{d_0/2}$ then
\begin{equation*}
\begin{split}
(\tilde{u}(x)-\tilde{u}(y))^2-({u}(x)-{u}(y))^2&=(t-\ov c \beta \psi(x)-u(x))(t-\ov{c} \beta \psi(x)
+u(x)-2u(y))\\
&\le 2(1- \psi(x))(\ov c\beta )^2(\psi(y)-\psi(x)),
\end{split}
\end{equation*}
\item if $x\in K$ and $y \in K^c \cap B_{d_0/2}^c=B_{d_0/2}^c$ then
\begin{equation*}
\begin{split}
(\tilde{u}(x)-\tilde{u}(y))^2-({u}(x)-{u}(y))^2&=(t-\ov c \beta \psi(x)-u(x))(t-\ov{c} \beta \psi(x)
+u(x)-2u(y))\\
&\le (t-\ov c \beta \psi(x)-u(x))^2 +2\overline c \beta(1-\psi(x))(u(x)-u(y))\\
&\le (\ov c \beta)^2(1-\psi(x))^2+2\overline c \beta(1-\psi(x))(u(x)-u(y)).
\end{split}
\end{equation*}
\end{enumerate}
Notice also that
\begin{align*}
&\int_{\R^n}\int_{\R^n}\fr{(\tilde{u}(x)-\tilde{u}(y))^2}{|x-y|^{n+2s}}\dd x \dd y-\int_{\R^n}\int_{
\R^n}\fr{(u(x)-u(y))^2}{|x-y|^{n+2s}}\dd x \dd y\\
&\le (\ov{c}\beta)^2\int_K \int_K \fr{|\psi(x)-\psi(y)|^2}{|x-y|^{n+2s}}\dd y \dd x\\
&+4(\ov c \beta)^2 \int_K \int_{
K^c\cap B_{d_0/2}}\fr{(1-\psi(x))(\psi(y)-\psi(x))}{|x-y|^{n+2s}}\dd y \dd x\\
&\qu +2(\ov{c}\beta)^2 \int_K\int_{B_{d_0/2}^c}\fr{(\psi(y)-\psi(x))^2}{|x-y|^{n+2s}}\dd y \dd x \\
&+4\ov c \beta
\int_K\int_{B_{d_0/2}^c}\fr{(1-\psi(x))(u(x)-u(y))}{|x-y|^{n+2s}}\dd y \dd x\\
&=: I_1+I_2+I_3+I_4.
\end{align*}
By definition of $\psi$, we see that $\tilde\psi(x):=\psi(d_0x/2)$ is a radial cut-off function
which is independent of $d_0$. Thus we have
\begin{align*}
I_1+I_2+I_3 \le C \beta^2 d_0^{n-2s},
\end{align*}
where the constant $C$ does not depend on $\beta$ and $d_0$. To estimate $I_4$, we may assume that
$1-\psi(x)=O((d_0/2-|x|)^2)$ and that $\norm{u}_{L^\infty}=1$, which implies
\begin{align*}
\begin{split}
I_4 &\le C \beta  \int_K\int_{B_{d_0/2}^c}\fr{(d_0/2-|x|)^2}{|x-y|^{n+2s}}\dd y \dd x\le C \beta
\int_K\int_{B_{d_0/2-|x|}^c(x)}\fr{(d_0/2-|x|)^2}{|x-y|^{n+2s}}\dd y \dd x\\
&\le C \beta  \int_K (d_0/2-|x|)^{2-2s}\dd x \le C \beta  d_0^{n+2-2s}.
\end{split}
\end{align*}
Combining these facts, together with \eqref{ineq-minimality}, we obtain
\begin{align*}
(\La-\al) |B_{d_0/4}|(2t\underline{c}\beta-\underline{c}^2\beta^2 )\le C\beta^2 d_0^{n-2s}+C\beta
\norm{u}_{L^\infty} d_0^{n+2-2s}.
\end{align*}
By the optimal regularity, we can take small $c_0$ so that $\beta\le \norm{u}_{C^{2s}}d_0^{2s}\le \underline c t$, which gives
\begin{align*}
(\La-\al) t\underline{c}\le C\beta d_0^{-2s}+C \norm{u}_{L^\infty} d_0^{2-2s}.
\end{align*}
Again, taking small $c_0$, we conclude that
\begin{align*}
\beta \ge C_0 d_0^{2s}
\end{align*}
for some constant $C_0$. This completes the proof of \eqref{nondeg-1}. Since the proof of
\eqref{nondeg-2} is similar to that of \eqref{nondeg-1}, we omit the details here.
\end{proof}
Using the previous Lemma, together with a blowup argument, we are able to  show nondegeneracy. Let us denote the nearest
point to $x$ in $\pa D$ by $\tilde x$ so that $\dist(x,\pa D)=\dist(x,\tilde x)$. The argument follows
as in \cite{CRS}:

\begin{lemma}\label{lem-nondeg}
Let $(u,D)$ be an optimal pair and $a>0$. Take $x_0$ to be a point in $\pa D$.
Then there is a constant $C$, independent of $u$, such that
\begin{align*}
\sup_{B_r(x_0)}|t-u|\ge C r^{2s}.
\end{align*}
\end{lemma}

\begin{proof}
Let $x_0\in \pa D$ and $B_r(x_0)\su \Om$. Let $x_1\in B_{r}(x_0)$ such that $u(x_1)<t$ and
$d_1:=\dist(x_1,\pa D)<c_0$ where the constant $c_0$ is defined in \Cref{lem-nondeg-pre}. By the same Lemma,
\begin{align*}
\tau:=\fr{t-u(x_1)}{d_1^{2s}}\ge C_0.
\end{align*}
We claim that there exist constants $\de>0$ and $M>0$ which are independent of $x_1$ and such that
\begin{align*}
\sup_{B_{Md_1}(\tilde x_1)}(t-u(x)) \ge (1+\de) \tau d_1^{2s}.
\end{align*}
If not, we can take a sequence $x_k$ with $d_k:=\dist(x_k,\pa D)$ so that
\begin{align}\label{ineq-contra-seq}
\sup_{B_{k d_k}(\tilde x_k)}(t-u(x))\le \left(1+\fr{1}{k}\right) \tau d_k^{2s}.
\end{align}
Now we define
\begin{align*}
v_{k,\tilde x_k}(x)=\fr{t-u(\tilde x_k+\dist(x,\pa D)x)}{\dist(x,\pa D)^{2s}}.
\end{align*}
In terms of $v_{k,\tilde x_k}$, \eqref{ineq-contra-seq} becomes
\begin{align*}
\sup_{B_k(0)}v_{k,\tilde x_k} \le \left( 1+\fr{1}{k}\right) \tau.
\end{align*}
Then, passing to the limit, we have a limit $v_0$ satisfying $\sup_{\R^n}v_0\le \tau$, $v_0(0)=0$, and
$v_0(z)=\tau>0$ for some $|z|=1$, which is a contradiction.

Using the claim, we construct a sequence $\{x_j\}$ such that $|x_j-x_{j-1}|\le (M+1)d_{j-1}$ and
$t-u(x_j) \ge (1+\de) (t-u(x_{j-1}))$. Since $\de>0$ does not depend on $x_j$, we deduce that there is
an index $j$ so that $x_j$ exits from $B_r(x_0)$. Assume that $j_0$ is the first index such that
$x_{j_0}\in B_r(x_0)$ and $x_{j_0+1}\not\in B_r(x_0)$. Then we have
\begin{align*}
t-u(x_{j_0})=\sum _{j\le j_0}(t-u(x_j)-(t-u(x_{j-1})))\ge\de \sum_{j\le j_0}(t-u(x_{j-1})).
\end{align*}
By \Cref{lem-nondeg-pre}, we see that $t-u(x_{j-1}) \ge C_0  d_{j-1}^{2s}\ge
C_0(M+1)^{-2s}|x_j-x_{j-1}|^{2s}$. If $2s\le1$, then we have
\begin{align*}
\sum_{j\le j_0}|x_j-x_{j-1}|^{2s}\ge \left(\sum_{j\le j_0}
(x_j-x_{j-1})\right)^{2s}=|x_{j_0}-x_0|^{2s}\ge r^{2s}.
\end{align*}
Combining these facts, we conclude that
\begin{align*}
\sup_{B_{Mr}} (t-u(x))\ge t-u(x_{j_0})\ge C r^{2s}
\end{align*}
and, therefore, the desired estimate is obtained by replacing $r$ by $r/M$. In a similar way, if we have
a point $x_1 \in B_r(x_0)$ such that $u(x_1)>t$, we obtain $\sup_{B_r}(u(x)-t)\ge Cr^{2s}$. This
completes the proof.
\end{proof}

We remark that in the case of $a=0$ the proof of \Cref{lem-nondeg}  also holds for the points at which
the pointwise $C^{0,1}$ norm is bounded. In summary, we have the following results on non-triviality for blowups.

\begin{corollary}\label{cor-nontrivial-a}
Let $(u,D)$ be an optimal pair and $a>0$. If a function $v_0$ is any blowup of $v:=t-u$, then $v_0$ is
not trivial, i.e., $v_0\not\equiv0$ on $\R^n$.
\end{corollary}

\begin{corollary}\label{cor-nontrivial-0}
Let $(u,D)$ be an optimal pair and $a=0$. Then either any convergent subsequence of the rescaled function $v_r$ or that of $v_r/C_r$ has a non-trivial limit, where $C_r$ is the quantity in \Cref{lem-class}.
\end{corollary}

\section{Structure of optimal configuration}\label{sec:struc}

In this section, we shall prove the equation \eqref{thmeq-D}, i.e., the optimal configuration $D$ is
given by the sublevel set of $u$, for $s\leq 1/2$ $(a\ge0)$. The results from the previous section are the key ingredients. We follow
the argument in \cite{FF}.

\begin{lemma}\label{lem-D-main}
Let $(u,D)$ be an optimal pair and $a \ge 0$. The optimal configuration $D$ is given by the sublevel set of optimal solution $u$, i.e.,
\begin{align*}
D = \{ x\in \Om : u(x) \le t\}.
\end{align*}
\end{lemma}

\begin{proof}
Since $D$ satisfies \eqref{eq-D2}, it suffices to show that the $t$-level set $\Ga_t$ of $u$ has measure
zero. Assume that $\Ga_t$ has positive measure. By Lebesgue's density theorem,
\begin{align*}
\chi_{\Ga_t}(x)= \lim_{r\ra 0+} \fr{1}{|B_r(x)|}\int_{B_r(x)}\chi_{\Ga_t}
\end{align*}
for all $x\in \R^n\sm N$, where $|N|=0$. This implies that, for all $x\in \Ga_t\sm N$,
\begin{align*}\label{eq-ae}
0=\lim_{r\ra 0+} \fr{|B_r(x)\cap (\R^n\setminus \Ga_t)|}{|B_r(x)|}.
\end{align*}

Now we fix a point $x_0\in \Ga_t\sm N$. For any $\e>0$ there exists $r_0$ such that, if $0<r<r_0$, then
\begin{align*}
|B_r(x_0)\cap  (\R^n\setminus \Ga_t)| \le \e |B_r(x_0)|.
\end{align*}
We recall $v(x)=t-u(x)$ from the previous section and observe $v\equiv 0 $ in $\Ga_t$. Then we have
\begin{align*}
\int_{B_r(x_0)}v^2\dd x&=\int_{B_r(x_0)\cap (\R^n\setminus \Ga_t)}v^2\dd x\\
&\le |B_r(x_0)\cap  (\R^n\setminus \Ga_t)|^{1-\fr{2}{2^*}}\left(\int_{B_r(x_0)} v^{2^*}\dd x\right)^{2/2^*}\\
&\le\e^{1-\fr{2}{2^*}} |B_r(x_0)|^{1-\fr{2}{2^*}}\left(\int_{B_r(x_0)} v^{2^*}\dd x\right)^{2/2^*},
\end{align*}
where $2^*=\fr{2n}{n-2s}$. In the notation of Section \ref{sec:blowups}, recall that we have defined $v_r(x)=r^{-2s}v(x_0+rx)$ so that the above inequality becomes
\begin{align*}
\int_{B_1(0)}v_r^2 \dd x\le \e^{1-\fr{2}{2^*}} |B_1(0)|^{1-\fr{2}{2^*}}\left(\int_{B_1(0)}
v_r^{2^*}\dd x\right)^{2/2^*}.
\end{align*}
Note that this still holds for $v_r/C_r$ instead of $v_r$. Then, \Cref{cor-nontrivial-a} and
\Cref{cor-nontrivial-0} imply that we have a subsequence $\{r_k\}$ such that either $\{v_{r_k}\}$ or $\{\tilde
v_{r_k}\}$ converge to a nonzero function $v_0$ or $\tilde v_0$ as $k\ra \infty$, respectively. By
taking $k\ra \infty$ in the above inequality we obtain, either for $\overline v=v_0$ or for $v=\tilde v_0$,
\begin{align*}
\int_{B_1(0)}\overline v^2 \dd x\le \e^{1-\fr{2}{2^*}} |B_1(0)|^{1-\fr{2}{2^*}}\left(\int_{B_1(0)}
\overline v^{2^*}\dd x\right)^{2/2^*},
\end{align*}
and then, by taking $\e\ra 0$, we finally have $v_0 \equiv0$ or $\tilde v_0\equiv0$ in $B_1(0)$,
respectively, which is a contradiction. Therefore, $|\Ga_t| =0$.
\end{proof}

We are now ready to prove our main Theorem:

\begin{proof}[Proof of \Cref{thm-reg}]
The regularity assertions follow from both \Cref{lem-regularity} and \Cref{lem-optreg}. Now \Cref{lem-D} and
\Cref{lem-D-main} give the sublevel set property.\\
\end{proof}

Next, we investigate some properties of the optimal configuration $D$ for general $\al>0$. To do
this, we begin with the following Lemma which implies the continuity of $(-\De)^su$.
\begin{lemma}\label{lem-continuity}
Let $u\in H^{2s}(\Om) \cap L^\infty(\Om)$. If $u$ is locally constant near a point $x_0\in \Om$, then
$(-\De )^s u(x)$ is continuous at $x_0$.
\end{lemma}

\begin{proof}
Let $u$ be  constant in $B_r(x_0)$ for some $r>0$, and take a sequence $\{x_k\}\subset
B_{\fr{r}{2}}(x_0)$ converging to $x_0$. Notice that, for $\rho_k= r- |x_k-x_0|$,
$B_{\rho_k}(x_k)\subset B_r(x_0)$, and thus
\begin{equation*}
\begin{split}
\int_{\R^n\setminus B_r(x_0)}\fr{u(x_k)-u(y)}{|x_k-y|^{n+2s}}\dd y&\le 2\norm{u}_{L^\infty} \int _
{\R^n\setminus B_{\rho_k}(x_k)} \fr{1}{|x_k-y|^{n+2s}}\dd y\\
&=2\norm{u}_{L^\infty} \fr{n\om_n\rho_k^{-2s}}{2s}\\
&\le \fr{n\om_n 2^{2s}\norm{u}_{L^\infty}}{sr^{2s}}.
\end{split}
\end{equation*}
Using this and Lebesgue dominated convergence theorem, we conclude that
\begin{align*}
\lim _{k\ra \infty }(-\De)^{s}u (x_k)= c_{n,s}\int _{\R^n\setminus B_r(x_0)}
\fr{u(x_0)-u(y)}{|x_0-y|^{n+2s}}\dd y= (-\De )^{s}u(x_0).
\end{align*}
\end{proof}

We remark here that our regularity statement in \Cref{lem-regularity} does not imply the continuity of $(-\De)^su$.

\begin{corollary}\label{cor-D}
Let $(u,D)$ be an optimal pair. If $u$ is locally constant near a point $x_0\in \Om$, then either
$x_0\in \trm{int}\ ( D) $ or $x_0\in \trm{int}\ ( \Om\setminus D)$ holds.
\end{corollary}

\begin{proof}
From \Cref{lem-continuity}, $(-\De)^s u$ is continuous at $x_0$, and therefore so is $(\La-\al\chi_D)u$.
Since $u$ is a continuous function, there is a neighborhood $U$ of $x_0$ such that $U\subset D$ or
$U\subset \Om\setminus D$ hold. This completes the proof.
\end{proof}

The last Lemma in this section asserts that the any level set $\{ u=c\} $ does not have an interior point
in $\Om$ if $c>0$. In particular, $\{u=t\}$ has no interior points.

\begin{lemma}\label{lem-notconst}
Let $(u,D)$ be an optimal pair. Then $u$ is not locally constant near any point in $\Om$.
\end{lemma}

\begin{proof}
Assume that $u$ is a locally constant near a point $x_0$ in $\Om$. From \Cref{cor-D}, $(\La-\al\chi_D)u$
is a locally constant function. Then using the unique continuation property (see \cite{FF}), we have
$u\equiv 0$, which yields a contradiction.
\end{proof}

\section{Symmetry property}\label{sec:symm}
We devote  this section in proving a symmetry property when the domain has a directional symmetry
and convexity. The basic idea is to use Steiner symmetrization, but there is a technical issue when we
consider the equality case. To overcome this, we slightly modify the kernel a little bit.

\begin{proof}[Proof of \Cref{thm-symm}]
As in the local case from \cite{CGI+}, we apply Steiner symmetrization to the function $u(\cdot,x')$ and the set
$\{x_1:(x_1,x')\in D\}$ for each $x'=(x_2,\cdots,x_n)$. We refer the reader to Chapter 3 in \cite{LL}
for the definition and various properties of Steiner symmetrization.

Let $u^*(\cdot,x')$ be the Steiner symmetrization of $u(\cdot,x')$ for each $x'$, namely the function
$u^*$ is symmetric in $x_1$ and decreasing for $x_1\geq0$ with the same measure of super level set
\begin{align*}
	|\{x_1:u^*(x_1,x')>t\}|=|\{x_1:u(x_1,x')>t\}|.
\end{align*}
Using the integral representation, $\int f\dd x=\int_0^\infty \{f>t\}\dd t$, an easy consequence of the
definition is that $\int_\R (u^*)^2 \dd x_1=\int_\R u^2 \dd x_1$. Thus, integrating in $x'$,
\begin{align}\label{eq:re-1}
	\int_\Om (u^*)^2 \dd x=\int_\Om u^2 \dd x.
\end{align}
It is also a well-known property that $\int_\R \chi_{D^c}u^2 \dd x_1 \leq \int_\R
(\chi_{D^c})^*(u^*)^2 \dd x_1$, which is equivalent to
\begin{align}\label{eq:re-22}
	\int_\R (\chi_{D_*})(u^*)^2 \dd x_1 \leq \int_\R \chi_{D}u^2 \dd x_1,
\end{align}
where the set $D_*$ is defined by $\chi_{D_*}=1-(\chi_{D^c})^*$. Again, we integrate
\eqref{eq:re-22} in $x'$ to obtain
\begin{align}\label{eq:re-2}
	\int_\Om (\chi_{D_*})(u^*)^2 \dd x \leq \int_\Om \chi_{D}u^2 \dd x.
\end{align}

Now we claim that
\begin{align}\label{claim:re}
	\int_{\R^n}\int_{\R^n}\fr{(u^*(x)-u^*(y))^2}{|x-y|^{n+2s}}\dd x \dd y \leq
\int_{\R^n}\int_{\R^n}\fr{(u(x)-u(y))^2}{|x-y|^{n+2s}}\dd x \dd y,
\end{align}
and that this inequality holds with equality if and only if $u=u^*$. To see this, consider an approximate kernel
\begin{align*}
	K_\e(x_1;x')=(|x_1|^2+|x'|^2+\e)^{-\fr{n+2s}{2}}
\end{align*}
for $\e>0$ and note that $\norm{K_\e(\cdot;x')}_{L^1(\R)} \leq \norm{K_\e(\cdot;0)}_{L^1(\R)}\leq C$. It
follows from Theorem 3.7 in \cite{LL} that
\begin{equation}\label{ineq:re}\begin{split}
	\int_{\R}\int_{\R}&(u(x)-u(y))^2K_\e(x_1-y_1;x'-y')\dd x_1 \dd y_1\\ &=2\int_\R
u(y)^2\norm{K_\e(\cdot;x'-y')}_{L^1(\R)}\dd x_1\\
&\qquad-2\int_\R\int_\R u(x)u(y)K_\e(x_1-y_1;x'-y')\dd x_1\dd
y_1\\
	&\geq2\norm{K_\e(\cdot;x'-y')}_{L^1(\R)}\int_\R (u^*)^2\dd x_1\\
&\qquad -2\int_\R\int_\R
u^*(x)u^*(y)K_\e(x_1-y_1;x'-y')\dd x_1\dd y_1\\
	&=\int_{\R}\int_{\R}(u^*(x)-u^*(y))^2K_\e(x_1-y_1;x'-y')\dd x_1 \dd y_1.
\end{split}\end{equation}
Since the both hand sides of this inequality converge to the claimed expressions as $\e \ra 0$, the
result follows by monotone convergence theorem and integrating in $(x',y')$.

To verify the equality condition in \eqref{claim:re}, first note that if $x'\not=y'$, then the
inequality \eqref{ineq:re} holds even for $\e=0$ since $\norm{K_0(\cdot;x'-y')}_{L^1(\R)}\leq C$. Thus,
the equality in \eqref{claim:re} implies that in \eqref{ineq:re} with $x'\not=y'$ and $\e=0$. Then, from
Theorem 3.9 in \cite{LL}, $u(x)=u^*(x_1-z,x')$, where $z\in \R$ depends on $x'$. The number $z$, however,
must be zero because $u$ is symmetric with respect to the hyperplane $\{x_1=0\}$. Thus $u\equiv u^*$, as
claimed.

Now we are ready to prove assertions. Since the eigenvalue $\la(\al,D)$ is given by
\begin{align*}
	\inf_{u}\fr{\norm{(-\Delta)^{s/2}u}^2+\al\int_D u^2\dd x}{\int_\Om u^2 \dd x},
\end{align*}
we have from \eqref{eq:re-1}, \eqref{eq:re-2}, and \eqref{claim:re} that $\la(\al,D_*)\leq \la(\al,D)$.
Therefore, if $(u,D)$ is an optimal pair, then $\La(\al,D)=\la(\al,D_*)$ and, by the equality condition
in \eqref{claim:re}, $u\equiv u^*$. This proves the Theorem.
\end{proof}

Using this, we can show that the optimal configuration is an annulus when the domain is a ball.

\begin{proof}[Proof of \Cref{cor-ball}]
From \Cref{thm-symm}, $u$ is a rotationally symmetric function and decreases in the radial direction.
Moreover, \Cref{lem-D-main} implies \eqref{D-radial} and then the strictly decreasing property follows
from \Cref{lem-notconst}. To prove the uniqueness assertion, we first note that $r(A)$ does not depend
on $u$ so that the optimal configuration $D$ is unique. Now we assume that there are two solutions $u_1$
and $u_2$ with $t_1$ and $t_2$ such that
\begin{align*}
D= \{ u_1 \le t_1\} =\{ u_2\le t_2\}.
\end{align*}
Define $v:=u_1/t_1 - u_2/t_2$ and notice that $v$ solves
\begin{equation*}
\begin{alignedat}{3}
(-\De)^s v + \al \chi_D v&= \La v &&\qu \trm{in } \Om,\\
v &=0 &&\qu \trm{on } \R^n\sm\Om.
\end{alignedat}
\end{equation*}
From the definition of $\La$, together with \Cref{lem-sign} and \Cref{PN relation}, we can see that $v$
has a sign in $\Om$. This is a contradiction since $v(x)=0$ for $ |x|=r(A)$.
\end{proof}

\section{symmetry breaking}\label{sec:symm-breaking}
In the previous section, we proved the symmetry property of the optimal pair $(u,D)$ when the domain has
directional symmetry and convexity. Here we construct an example which presents  symmetry
breaking when the domain has radial symmetry for the case $s< \2$. An immediate consequence is non-uniqueness. Notice that the ball is the only case we can prove the uniqueness.

In \cite{CGI+}, the authors give the symmetry breaking examples when the domains are an annulus and a dumbbell
shape for the (local) composite membrane problem. For the nonlocal equation, N\'apoli considered  in \cite{Napoli} the
symmetry breaking property for an elliptic equation involving the fractional Laplacian. In that work,
the author proved that there are both a nontrivial radial solution and a non-radial one for a nonlocal
elliptic problem. Note that we shall prove that, for some large annular domain, the nonlocal
composite membrane problem admits only non-radially symmetric solutions.

We follow the argument in \cite{CGI+}, considering the third eigenvalue problem connecting the radial
symmetry eigenvalue problem to the non-radial one. However, some difficulties occur from the
nonlocality; for example, it is not clear how to decompose the fractional Laplacian into radial and
angular parts.

For an annulus
\begin{align*}
\Om_b=\{ x\in \R^2 ; b<|x|< b+1 \} ,\quad b>0,
\end{align*}
and a radial subset $D$ in $\Om_b$ such that
\begin{align}\label{D=,D_1=}
D=\{ (r,\theta ) ;	r \in D_1 , 0\le \theta <2\pi \} ,\quad D_1\subsetneq (b,b+1),
\end{align}
we consider the eigenvalue problem of the form
\begin{align}\label{evp-symm}
\begin{cases}
(-\De)^s u+\al \chi_D u=\sigma u \quad \trm{in }  \Om_b, \\
u=0 \quad  \trm{on }  \R^n \setminus \Om_b,
\end{cases}
\end{align}
where $u$ and $\sigma$ are the first eigenfunction and eigenvalue, respectively.

We shall construct a function $\tilde u$ and a domain $\tilde D$ with $|\tilde D|=|D|$, which satisfy
\begin{align*}
\fr{\int_{\Om_b}\tilde u(-\De )^s \tilde u \dd x+\al\int_{\Om_b}\chi_{\tilde D}\tilde u^2 \dd x}{\int_{\Om_b}\tilde
u^2\dd x} < \sigma.
\end{align*}
This means that any domain $D$ having symmetry is not an optimal configuration.

Let $\de=|D| / |\Om_b|$ and take a number $N=N(\de)$ such that
\begin{align*}
\de < 1-\fr{1}{2N}.
\end{align*}
To construct $(\tilde u, \tilde D)$, we define the sector
\begin{align*}
E_+=\Om_b\cap \{ (r,\theta); 0 \le \theta \le \pi / N\}.
\end{align*}
Then we may choose $\tilde D \subset \Om_b\setminus E_+$ since $|\tilde D|=\de|\Om_b|
<(1-\fr{1}{2N})|\Om_b|=|\Om_b \setminus E_+|$.

Let $\tilde u$ be the first Dirichlet eigenfunction of the fractional Laplacian on $E_+$ and $\la _1
(E_+)$ be the first eigenvalue so that
\begin{align*}
(-\De)^s \tilde u =&\la_1 (E_+ )\tilde u \quad \trm{in} \quad E_+,\\
\tilde u =&0 \quad \trm{on}\quad  \R^n \setminus E_+.
\end{align*}
Note that since $\trm{supp}\ \tilde u\cap \tilde D=\emptyset$, it is enough to show that
\begin{align}\label{lasi}
\la_1 (E_+)<\sigma.
\end{align}
In order to prove this, we need to introduce an intermediate eigenvalue problem. Define $v_0$ to be the
lowest eigenfunction for \eqref{evp-symm} among functions of the form
\begin{align*}
v(r,\theta)=h(r)\sin N\theta,
\end{align*}
and let $\tau$ be the associated eigenvalue. Clearly, $\si\le \tau$. We claim that $\tau$ is close
enough to $\si$ when $b$ is large. \\

\textbf{Claim 1.} $\tau \le \si +O(b^{-1-2s})$ as $b\ra \infty$.\\

Notice that $-1-2s > -2$ since $s < \2$. To prove \eqref{lasi}, we also need to show that $\la_1(E_+)$
is strictly less than $\tau$.\\

\textbf{Claim 2.} $\la_1(E_+)+c \le \tau$, where $c$ does not depend on $b$.\\

We will prove these claims below. Before this, we show \eqref{lasi} under the assumption that Claims 1
and 2 hold. By the claims, we have
\begin{align*}
\la_1(E_1)+c \le\tau \le \si +O(b^{-1-2s}).
\end{align*}
Taking large $b$, \eqref{lasi} follows.

\begin{proof}[Proof of Claim 1]
Let $h$ be the eigenfunction of \eqref{evp-symm} corresponding to the eigenvalue $\si$. Since $D$ has
radial symmetry, so does the eigenfunction $h$. Moreover, it is easy to see $h$ has a sign. Let $h>0$ in
$\Om_b$.

Take $v(r,\theta) = h(r)\sin N\theta$ in $\R^2$, and consider its extension $V$ to $\mathbb R^{3}_+$ given by \eqref{CS-extension}. Recall that the extended function $V$  is given by
\begin{align*}
V(x,y)= (P(\cdot ,y) * v)(x),
\end{align*}
where $P$ is the Poisson kernel from \eqref{Poisson}.
Since the Poisson kernel is a rotationally symmetric function and $v$ has a special form, $V$ also has such
special form. More precisely, we have
\begin{align*}
V(R_\vp x,y)=& \int_{\R^2}P(R_\vp x-\xi,y)v(\xi)\dd\xi \\
=& \int_0^\infty \int _0 ^{2\pi} P(R_\vp x-t(\cos \theta,\sin \theta),y) h(t)\sin (N\theta) t \dd\theta
\dd t\\
=& \int_0^\infty \int _0 ^{2\pi} P(x-t(\cos( \theta-\vp),\sin (\theta-\vp)),y) h(t)\sin (N\theta) t
\dd\theta \dd t\\
=& \int_0^\infty \int _0 ^{2\pi} P(x-t(\cos\theta,\sin \theta),y) h(t)\sin (N\theta+N\vp) t \dd\theta
\dd t\\
=& \cos(N\vp)\int_0^\infty \int _0 ^{2\pi} P(x-t(\cos\theta,\sin \theta),y) h(t)\sin (N\theta) t
\dd\theta \dd t\\
&+\sin(N\vp) \int_0^\infty \int _0 ^{2\pi} P(x-t(\cos\theta,\sin \theta),y) h(t)\cos (N\theta) t
\dd\theta \dd t\\
=&\cos (N\vp) V( x,y)+\sin (N\vp) V(R_{\fr{\pi}{2N}} x,y),
\end{align*}
where $R_\vp$ denotes a rotation and $(t,\theta)$ are polar coordinates for $\xi$. Notice that $V\equiv 0$ on $\{\theta=0\}$. Therefore, the extended
function $V$ also has the form
\begin{align*}
V( x, y ) =H(r,y)\sin (N\theta).
\end{align*}
Using this, the extension problem \eqref{CS-extension} for $V$ becomes
\begin{equation}\label{eq-H}\begin{split}
L_a H=&\fr{N^2}{r^2} H \quad \trm{on} \quad \R \times \{y>0\},\\
H(r,0)=&h(r) \quad \trm{on} \quad \R.
\end{split}
\end{equation}
Moreover, we have
\begin{equation}\label{def-H}
\begin{split}
H(r,y) &= V(R_{\fr{\pi}{2N}}(r,0),y)\\
&= \int_0^\infty \int _0 ^{2\pi} P((r,0)-t(\cos\theta,\sin \theta),y) h(t)\cos (N\theta) t \dd \theta
\dd t\\
&=C_{2,s} y^{2s}\int_b^{b+1}\int_0^{2\pi} \fr{th(t)\cos N\theta}{(r^2+t^2+y^2-2rt\cos \theta)^{1+s}}\dd
\theta \dd t.
\end{split}
\end{equation}
These properties yield the following Lemma:

\begin{lemma}\label{lem-H}
Let $\tilde H$ be the extended function of $h$. Then we have
\begin{align*}
0 \le H \le \tilde H.
\end{align*}
\end{lemma}	

\begin{proof}
Since $\tilde H(\cdot, y)= P(\cdot,y)\ast h$, the second inequality follows from $h>0$ on $(b,b+1)$ and
$\cos(N\theta)\le 1$.

To see the first inequality, we observe that for any $\e>0$, there exists $R=R(\e)>0$ such that
\begin{align}\label{H-e}
|H(r,y)| < \e \qu \trm{on } \R^{n+1}_+\setminus B_R^+.
\end{align}
In fact, if $\sqrt{r^2+y^2} \ge 2b+2$, then we see
\begin{align*}
(r-t)^2+y^2 \ge (\sqrt{r^2+y^2}-t)^2\ge \fr{1}{4}(r^2+y^2),
\end{align*}
and hence, together with the Cauchy-Schwarz inequality and \eqref{def-H}, we have that
\begin{align*}
|H(r,y)|&\le 2\pi C_{2,s}y^{2s}\int_b^{b+1}\fr{ t h(t)}{\big((r-t)^2+y^2\big)^{1+s}}\dd t\\
&\le\fr{2^{3+2s}\pi C_{2,s}y^{2s}}{(r^2+y^2)^{1+s}}\sqrt{\left(\int_b^{b+1} h(t)^2t\dd
t\right)\left(\int_b^{b+1} t \dd t \right)}\\
&\le \fr{2^{2+2s}C_{2,s}\sqrt{ \pi(2b+1)}  \norm{h}_{L^2(\Om_b)}}{r^2+ y^2 } \ra 0 \qu \trm{as
}r^2+y^2\ra \infty.
\end{align*}
This gives \eqref{H-e}.

Now we assume, by contradiction, that $H(r,y)=-2\e $ for some point $(r_0,y_0)$ where $\e>0$. Take $R=R(\e)$ as in the above,
and then
\begin{align}\label{eq-H1}
-2\e \ge \inf_{\R_+\ti \R_+} H(r,y) = \inf_{B_R^+}H(r,y).
\end{align}
However, by a simple maximum principle argument for equation \eqref{eq-H}, we have
\begin{align}\label{eq-H2}
\inf_{B_R^+} H(r,y) = \inf_{\Ga_R^+ \cup \Ga_R^0} H(r,y).
\end{align}
Since $H(r,0)=h(r)\ge 0$, \eqref{eq-H1} and \eqref{eq-H2} imply that
\begin{align*}
-2\e \ge \inf_{\Ga_R^+} H(r,y) \ge -\e,
\end{align*}
 which is a contradiction.
\end{proof}

To finish the proof of Claim 1, we now compare the two eigenvalues $\tau$ and $\sigma$. From the definition of
$\tau$ and $\si$, together  with \Cref{lem-H}, we have
\begin{equation}\label{def-tau}
\begin{split}
\tau&= \fr{\int_{\R_+\times \R_+}(H_r^2+H_y^2)y^a r \dd r  \dd y+ \al \int_{D_1} h^2 r\dd
r}{\int_b^{b+1} h^2r\dd r}+\fr{\int_{\R_+\times \R_+}\fr{N^2}{r^2}H^2y^a r \dd r \dd y}{\int_b^{b+1}
h^2r\dd r}\\
&\le \si+\fr{\int_{\R_+\times \R_+}\fr{N^2}{r}H^2y^a  \dd r \dd y}{\int_b^{b+1} h^2r\dd r}.
\end{split}
\end{equation}
Hence, it only remains to prove
\begin{align}\label{claim1-key}
\fr{\int_{\R_+\times \R_+}\fr{N^2}{r}H^2y^a  \dd r \dd y}{\int_b^{b+1} h^2r\dd r} = O(b^{-1-2s}) \qu
\trm{as }b\ra \infty.
\end{align}
To see this, we consider the following two quantities:
\begin{align*}
I_1 := \fr{\int_0^\infty \dd y\int_0^{\fr{b}{2}}\dd r\fr{N^2}{r}H^2y^a  }{\int_b^{b+1} h^2r\dd r}, \qu
I_2:=\fr{\int_0^\infty \dd y\int_{\fr{b}{2}}^\infty \dd r \fr{N^2}{r}H^2y^a  }{\int_b^{b+1} h^2r\dd r}.
\end{align*}
We first estimate $I_2$. Recall that $\sin \theta \ge \fr{2}{\pi} \theta$ if $0\le \theta\le
\fr{\pi}{2}$. Using this and expression \eqref{def-H}, we have
\begin{align*}
H(r,y)&=C_{2,s} y^{2s}\int_b^{b+1}\int_{-\pi}^{\pi} \fr{th(t)\cos N\theta}{\left((r-t)^2+y^2+4rt\sin^2
\fr{\theta}{2}\right)^{1+s}}\dd \theta \dd t\\
&\le C_{2,s} y^{2s}\int_b^{b+1}\int_{-\pi}^{\pi} \fr{th(t)}{\left((r-t)^2+y^2+\fr{4}{\pi^2}rt\theta^2
\right)^{1+s}}\dd \theta \dd t.
\end{align*}
Let $K = \sqrt{\fr{(r-t)^2+y^2}{\fr{4}{\pi^2}rt}}$ and $\theta=K\theta'$. Then we obtain
\begin{align*}
H(r,y)\le& 2C_{2,s} y^{2s}\int_b^{b+1}\int_{0}^{\fr{\pi}{K}}
\fr{Kth(t)}{\left((r-t)^2+y^2\right)^{1+s}\left(1+\theta'^2 \right)^{1+s}}\dd \theta' \dd t\\
\le&4C_{2,s} y^{2s}\int_b^{b+1}\fr{Kth(t)}{\left((r-t)^2+y^2\right)^{1+s}}\dd t,
\end{align*}
since $\int_0^\infty \fr{\dd \theta'}{(1+\theta'^2)^{1+s}}<2$. By the Cauchy-Schwarz inequality,
\begin{align*}
H(r,y)^2 &\le C(s) y^{4s}\int _b^{b+1}\fr{K^2t}{\left((r-t)^2+y^2\right)^{2+2s}}\dd t \int _b^{b+1}
h^2(t) t \dd t\\
&\le C(s) y^{4s}\int _b^{b+1}\fr{1}{r\left((r-t)^2+y^2\right)^{1+2s}}\dd t \int _b^{b+1} h^2(t) t \dd t.
\end{align*}
Therefore, we see that
\begin{align*}
I_2\le C(N,s)\int_0^\infty \dd y \int_{\fr{b}{2}}^\infty \dd r\int_b^{b+1}\dd t
\fr{y^{1+2s}}{r^2\left((r-t)^2+y^2\right)^{1+2s}},
\end{align*}
Now take $y=|r-t|y'$. Using the property
$\int_0^\infty  \fr{y'^{1+2s}}{\left(1+y'^2\right)^{1+2s}} \dd y'\le 1+\fr{1}{2s}$, we arrive to
\begin{align*}
I_2\le C(N,s) \int_b^{b+1}\dd t \left(\int_{\fr{b}{2}}^{t}+\int_{t}^{2b+1}+\int_{2b+1}^{\infty}\right)
\dd r \fr{1}{r^2(r-t)^{2s}}.
\end{align*}
We conclude that $I_2= O(b^{-1-2s})$ as $b\ra \infty$ by a direct computation.

Our next task is estimating $I_1$. Again, recall \eqref{def-H} so that we observe
\begin{align*}
\fr{H(r,y)}{C_{2,s}y^{2s}}&=\sum_{i=0}^{2N-1}\int_b^{b+1} \int_{\pi(2i-1)/2N}^{\pi (2i+1)/2N}
\fr{th(t)\cos(N\theta)}{(r^2+t^2+y^2-2rt\cos\theta)^{1+s}}\dd \theta \dd t\\
&= \fr{1}{N} \sum_{i=0}^{2N-1}\int_b^{b+1} \int_{0}^{\pi}
\fr{th(t)\cos(\pi(2i-1)/2+\vp)}{(r^2+t^2+y^2-2rt\cos(\vp/N+\pi(2i-1)/2N))^{1+s}}\dd \vp \dd t\\
&= \fr{1}{N} \sum_{i=0}^{2N-1}\int_b^{b+1} \int_{0}^{\pi} \fr{(-1)^i
th(t)\sin\vp}{\big(r^2+t^2+y^2-2rt\cos(\vp/N+\pi(2i-1)/2N)\big)^{1+s}}\dd \vp \dd t.
\end{align*}
By the mean value theorem, we see that
\begin{align*}	
\fr{H(r,y)}{C_{2,s}y^{2s}}&\le  \fr{1}{N} \sum_{i=0}^{2N-1}\int_b^{b+1} \int_{0}^{\pi} \fr{(-1)^i
th(t)\sin\vp}{(r^2+t^2+y^2-2(-1)^i rt)^{1+s}}\dd \vp \dd t\\
&=2\int_b^{b+1}  \left]\fr{th(t)}{\left((t-r)^2+y^2\right)^{1+s}}-\fr{th(t)}{\left((t+r)^2+y^2\right)^{1+s}}
\right]\dd t\\
&\le \fr{4(1+s)(3b+2)r}{\big(b^2/4+y^2\big)^{2+s}}\int_b^{b+1} th(t)\dd t
\end{align*}
for $r\in(0,b/2)$, which implies, using Cauchy-Schwarz as in the estimate for $I_2$, that
\begin{align*}
I_1 \le C(N,s)\int_0^{\fr{b}{2}} \dd r \int_0^\infty \dd y \fr{y^{1+2s}(b+1)^3r}{(b^2+y^2)^{4+2s}}.
\end{align*}
We finally take $y=by'$ and use $\int_0^\infty \fr{y'^{1+2s}}{(1+y'^2)^{4+2s}}\dd y'\le 2$ to conclude
$I_1=O(b^{-1-2s})$ as $b\ra \infty$. This completes the proof of Claim 1.
\end{proof}

In order to prove Claim 2 we need the following Lemma. Although it is true for any dimension $n$, we just
consider the two-dimensional case for simplicity.

\begin{lemma}\label{lem-symmrelation}
Let $N$ be any positive integer. Let $v$ be a function of the form $v(r,\theta)=h(r)\sin (N\theta)$ in $\Om_b$
with $v\equiv 0$ in $\R^2\sm \Om_b$ and $h(r)\ge 0$ for $r\in[b,b+1]$. Then we have
\begin{align*}
\norm{(-\De)^{s/2}v}^2_{L^2(\R^2)}\ge 2N \norm{(-\De)^{s/2}(v\chi_{E})}^2_{L^2(\R^2)},
\end{align*}
where $E=\Om_b\cap \{ (r,\theta): 0 \le \theta < \pi / N\}$.
\end{lemma}

\begin{proof}
In order to prove this, we first define
\begin{align*}
E_i=\Om_b \cap \{ (r,\theta):(i-1)\pi/N \le \theta < i\pi/N\}
\end{align*}
for $i=1,\cdots,2N$, and note that $E_1=E$.
Since $v$ is defined on $\Om_b= \cup_{i=1}^{2N}E_i$, we can decompose $v$ as $\sum_{i=1}^{2N}v_i$ where
$v_i=v\chi_{E_i}$. Observe that
\begin{align*}
|v(x)-v(y)|^2&=\left|\sum_{i=1}^{2N}\big(v_i(x)-v_i(y)\big)\right|^2 \\
&= \sum_{i=1}^{2N}\big(v_i(x)-v_i(y)\big)^2+\sum_{i\not=j}\big(v_i(x)-v_i(y)\big)\big(v_j(x)-v_j(y)\big)
\end{align*}
and $v(R_{k\pi /N}(x_1,x_2))=(-1)^{k}v(x_1,x_2)$. Using these, we have
\begin{align*}
\norm{(-\De)^{s/2}v}^2_{L^2(\R^2)}&= 2N \norm{(-\De)^{s/2}(v\chi_{E_1})}^2_{L^2(\R^2)}\\
&\qu +Nc_{2,s}\sum_{i\not=1} \int_{\R^2}\int_{\R^2}\fr{(v_1(x)-v_1(y))(v_i(x)-v_i(y))}{|x-y|^{2+2s}}\dd
x\dd y.
\end{align*}
We claim that the last term in the right hand side above is nonnegative. In fact, if one of $x$ and $y$
is contained in $\R^2\setminus \Om_b$ or both $x$ and $y$ are contained in the same $E_j$ for some $j$,
then
\begin{align}\label{eq-integrand-zero}
\left(v_1(x)-v_1(y)\right)\left(v_i(x)-v_i(y)\right)= 0
\end{align}
since $i\not=1$. Moreover, \eqref{eq-integrand-zero} also holds unless $(x,y)\in E_1\ti E_i$ or
$(x,y)\in E_i\ti E_1$. Thus, to have the conclusion, it suffices to show that
\begin{align}\label{ineq-sum-I}
-\sum_{i\not=1} \int_{E_1}\int_{E_i}\fr{v_1(x)v_i(y)}{|x-y|^{2+2s}}\dd x\dd y \ge 0.
\end{align}
To simplify the notation, let us define
\begin{align*}
I_i:=-\int_{E_1}\int_{E_i}\fr{v_1(x)v_i(y)}{|x-y|^{2+2s}}\dd x\dd y.
\end{align*}
Assume that $N=2k+1$. Notice that
\begin{align*}
I_{2k+2}&=-\int_{E_1}\int_{E_{2k+2}}\fr{v_1(x)v_{2k+2}(y)}{|x-y|^{2+2s}}\dd x\dd y\\
&=\int_{E_1}\int_{E_{1}}\fr{v_1(x)v_{1}(y)}{|x-R_{\pi}y|^{2+2s}}\dd x\dd y\ge 0,
\end{align*}
and
\begin{align*}
&I_2\ge -I_3\ge \cdots \ge -I_{2k+1},\\
&I_{4k+2}\ge-I_{4k+1}\ge \cdots \ge -I_{2k+3}.
\end{align*}
Then we have
\begin{align*}
\sum_{i=2}^{2N}I_i= \sum_{i=1}^k(I_{2i}+I_{2i+1})+\sum_{i=1}^k(I_{4k+4-2i}+I_{4k+3-2i})+I_{2k+2}\ge0.
\end{align*}
Now assume that $N=2k$. In this case, we see that
\begin{align*}
&I_2\ge -I_3\ge \cdots \ge I_{2k},\\
&I_{4k}\ge-I_{4k-1}\ge \cdots \ge I_{2k+2},\\
&I_{2k}+I_{2k+1}+I_{2k+2} \ge 0,
\end{align*}
which implies
\begin{align*}
\sum_{i=2}^{2N}I_i&=
\sum_{i=1}^{k-1}(I_{2i}+I_{2i+1})+\sum_{i=1}^{k-1}(I_{4k+2-2i}+I_{4k+1-2i})+I_{2k}+I_{2k+1}+I_{2k+2}\\
&\ge 0.
\end{align*}
In any case, we have \eqref{ineq-sum-I}, which completes the proof.
\end{proof}

To further proceed, we focus on the equation satisfied by the radial part $h_0$ of $v_0$. With some abuse
of notation, we write $h_0=h_0(|x|)=h_0(x)$, and then $(-\De)^s h_0$ is understood as fractional
Laplacian of the function $h_0$ defined on $\R^2$. Now we observe that for any $r$, by taking the point
$x$ such that $|x|=r$ and the angle of $x$ is $\fr{\pi}{2N}$, we have
\begin{align*}
(-\De)^s v_0(x)= (-\De)^s h_0(r) + c_{n,s}\int_0^{2\pi}\int_{b}^{b+1}\fr{h_0(t)(1-\sin
(N\theta))t}{(r^2+t^2-2rt\cos (\theta-\fr{\pi}{2N}))^{1+s}} \dd t\dd \theta.
\end{align*}
Moreover, for this $x$, the eigenfunction $v_1$ satisfies
\begin{align*}
(-\De)^s v_0(x)= (\tau -\al \chi_{D_1}(r))h_0(r).
\end{align*}
Therefore, the equation satisfied by $h_0$ is given by
\begin{align}\label{eq-h1}
(-\De)^s h_0(r)=(\tau-\al\chi_{D_1}(r))h_0(r) -B[h_0],
\end{align}
where
\begin{align*}
B[h]=c_{n,s}\int_0^{2\pi}\int_{b}^{b+1}\fr{h(t)(1-\sin (N\theta))t}{(r^2+t^2-2rt\cos
(\theta-\fr{\pi}{2N}))^{1+s}} \dd t\dd \theta.
\end{align*}

From now on, we estimate the coefficients in the right hand side of \eqref{eq-h1}.
\begin{lemma}\label{lem-estiB}
Let $0<s<\2$. Then we have, for $r\in [b,b+1]$,
\begin{align*}
B[h]\le C(s,N) b^{-1-2s} \norm{h}_{L^2(\Om_b)},
\end{align*}
where $C(s,N)$ is a constant.
\end{lemma}

\begin{proof}
We notice that
\begin{align*}
B[h]&= c_{2,s}\int_{-\pi}^{\pi}\int_{b}^{b+1}\fr{h(t)(1-\cos (N\theta))t}{(r^2+t^2-2rt\cos
\theta)^{1+s}} \dd t\dd \theta\\
&= 4c_{2,s}\int_{0}^{\pi}\int_{b}^{b+1}\fr{h(t)\sin^2 (\fr{N\theta}{2})t}{((r-t)^2+4rt\sin^2
(\fr{\theta}{2}))^{1+s}} \dd t\dd \theta.
\end{align*}
Since $\fr{2}{\pi}\theta\le \sin \theta $ for $0\le\theta\le\fr{\pi}{2}$ and $\sin\theta \le \theta$ for
any $\theta$, we have
\begin{align*}
B[h]\le C(s,N)\int_{0}^{\pi}\int_{b}^{b+1}\fr{h(t) \theta^2 t}{((r-t)^2+\fr{4}{\pi^2}rt \theta^2)^{1+s}}
\dd t\dd \theta.
\end{align*}
Now set $\theta = \kappa\theta' $ with $\kappa=\fr{|r-t|}{\sqrt{4rt/\pi^2}}$ so that
\begin{align*}
B[h]&\le C(s,N)\int_{0}^{\fr{\pi}{\kappa}}\int_{b}^{b+1}\fr{h(t) (\theta')^2
t\kappa^3}{(r-t)^{2+2s}(1+(\theta')^2)^{1+s}} \dd t\dd \theta'\\
&\le C(s,N)\int_{0}^{\fr{\pi}{\kappa}}\int_{b}^{b+1}\fr{h(t)
(\theta')^2|r-t|^{1-2s}}{b^2(1+(\theta')^2)^{1+s}} \dd t\dd \theta',
\end{align*}
where $C(s,N)$ is a constant depending only on $s,N$.
Observe that $\int_0^\infty \fr{\theta^2}{(1+\theta^2)^{1+s}}\dd \theta$ is the finite constant
depending on $s$ if $s<\2$. Using the Cauchy-Schwarz inequality, we therefore obtain
\begin{align*}
B[h] \le C(s,N)b^{-1-2s} \norm{h}_{L^2(\Om_b)}.
\end{align*}
\end{proof}

To estimate $\tau$, we need to estimate the energy of $h_0$ according to the dimension.
\begin{lemma}\label{lem-tau}
Assume that $b\ge 1$. Then the eigenvalue $\tau$ is bounded by some constant which is independent of
$b$.
\end{lemma}
\begin{proof}
From \eqref{def-tau} and \eqref{claim1-key}, we have
\begin{align*}
\tau \le \la_1+\al +O(b^{-1-2s}),
\end{align*}
where $\la_1$ is the first eigenvalue of $(-\De)^s$ in $\Om_b$. It suffices to show that $\la_1$ has a
uniform bound independent of $b$.

Let $h$ be a function defined in $\R^2$ with $h(x)=h(y)$ for any $|x|=|y|$ and $h(x)=0$ unless
$x\in\Om_b$. Writing $(-\De)^s_1$ for the $1$-dimensional fractional Laplacian, we shall compare to
$(-\De)^s _1 h$ and its first Dirichlet eigenvalue.

First, we observe that
\begin{align}\label{lem84-eq1}
2\pi b \norm{h}_{L^2(\R)}^2\le \norm{h}_{L^2(\R^2)}^2\le 2\pi (b+1) \norm{h}_{L^2(\R)}^2,
\end{align}
and
\begin{equation}\begin{split}\label{lem84-eq2}
\int_{\R^2}\int_{\R^2} \fr{|h(x)-h(y)|^2}{|x-y|^{2+2s}}\dd x\dd y&=
\int_0^\infty\int_0^\infty\int_{-\pi}^{\pi}   \fr{2\pi(h(r)-h(t))^2 rt}{((r-t)^2+4rt\sin ^2
(\fr{\theta}{2}))^{1+s}} \dd \theta \dd r \dd t \\
&=:I_{2,1}+2I_{2,2},
\end{split}
\end{equation}
where
\begin{align*}
I_{2,1}= \int_0^{2b+1} \int_0^{2b+1}\int_{-\pi}^{\pi}   \fr{2\pi(h(r)-h(t))^2 rt}{((r-t)^2+4rt\sin ^2
(\fr{\theta}{2}))^{1+s}} \dd \theta \dd r  \dd t
\end{align*}
and
\begin{align*}
I_{2,2}=   \int_{2b+1}^\infty \int_0^{2b+1} \int_{-\pi}^{\pi} \fr{2\pi h(r)^2 rt}{((r-t)^2+4rt\sin ^2
(\fr{\theta}{2}))^{1+s}} \dd \theta  \dd r \dd t.
\end{align*}
We also notice that
\begin{equation}
\begin{split}\label{lem84-eq3}
\int_{\R}\int_{\R} \fr{|h(x)-h(y)|^2}{|x-y|^{1+2s}}\dd x\dd
y&=\int_{-\infty}^\infty\int_{-\infty}^\infty\fr{(h(r)-h(t))^2}{|r-t|^{1+2s}}\dd r\dd t\\
&\le 4I_{1,1}+8I_{1,2},
\end{split}
\end{equation}
where
\begin{align*}
I_{1,1}=\int_0^{2b+1} \int_0^{2b+1}  \fr{(h(r)-h(t))^2}{|r-t|^{1+2s}}\dd r \dd t
\end{align*}
and
\begin{align*}
I_{1,2}=\int_{2b+1}^\infty \int _0^{2b+1} \fr{h(r)^2}{|r-t|^{1+2s}}\dd r\dd t
\end{align*}
We have
\begin{align*}
I_{2,1}\le C \int_0^{2b+1}\int_0^{2b+1} \int_{0}^{\pi}
\fr{(h(r)-h(t))^2}{((r-t)^2+(4rt/\pi^2)\theta^2)^{1+s}}rt  \dd \theta \dd r\dd t
\end{align*}
and for $K=\sqrt{\fr{(r-t)^2}{4\pi rt/\pi^2}}$, we substitute $\theta=K \theta'$ so that
\begin{equation*}\begin{split}
I_{2,1}&\le C \int_0^{\pi/K}\fr{1}{(1+\theta^2)^{1+s}}d\theta\int_0^{2b+1} \int_0^{2b+1}
\fr{(h(r)-h(t))^2}{|r-t|^{2+2s}}Krt  \dd r\dd t\\
&\le C \int_0^{2b+1} \int_0^{2b+1}   \fr{(h(r)-h(t))^2}{|r-t|^{1+2s}}\sqrt{rt}  \dd r\dd t.
\end{split}\end{equation*}
This implies $I_{2,1}\le  Cb I_{1,1}$.

Now we estimate $I_{2,2}$. Since $\supp{h}\subset [b,b+1]$, we have
\begin{equation*}\begin{split}
I_{2,2}&=   \int_{2b+1}^\infty \int_0^{b+1} \int_{-\pi}^{\pi} \fr{2\pi h(r)^2}{((r-t)^2+4rt\sin ^2
(\fr{\theta}{2}))^{1+s}}rt \dd \theta  \dd r \dd t\\
&\le C \int_{2b+1}^\infty \int_0^{b+1}  \fr{ h(r)^2 }{|r-t|^{1+2s}}\fr{rt}{t-r}   \dd r \dd t\\
&\le C b I_{1,2},
\end{split}\end{equation*}
where we have used $\fr{t}{t-r}\le 3$.

From \eqref{lem84-eq1}, \eqref{lem84-eq2}, and \eqref{lem84-eq3}, together with the above estimate, we
have
\begin{equation*}
\la_1 \le \fr{I_{2,1}+2I_{2,2}}{\norm{h}^2_{L^2(\R^2)}}\le
C\fr{4I_{1,1}+8I_{1,2}}{\norm{h}_{L^2(\R)}^2}\le C\fr{\int_{\R}\int_{\R}
\fr{|h(x)-h(y)|^2}{|x-y|^{1+2s}}\dd x\dd y}{\norm{h}_{L^2(\R)}^2}.
\end{equation*}
If we take $h$ to be the first eigenfunction defined on $(b,b+1)$, then the last quantity exists. Now the
conclusion follows from the fact that this  quantity does not depend on $b$.
\end{proof}

Now we are ready to prove the Lemma below in analogy to Lemma 15 in \cite{CGI+}.

\begin{lemma}\label{lem-alphaqubdd}
Assume that $b\ge 1$ and  let $v_0=h_0(r)\sin (N\theta)$ be the first eigenfunction corresponding to the
eigenvalue $\tau$ in $\Om_b$. Let $\de:=|D|/|\Om_b|$. Then
\begin{align*}
\int_{\Om_b}\chi_{D} v_0^2 \dd x \ge c \int_{\Om_b} v_0^2\dd x,
\end{align*}
where $c$ does not depend on $b$.
\end{lemma}

\begin{proof}
From \eqref{D=,D_1=}, since we can take $|D_1|=\de$, we have
\begin{align*}
\left|[b+\de/4,b+1-\de/4]\cap D_1\right|\ge \fr{\de}{2}.
\end{align*}
Then we have
\begin{align}\label{lem85-con1}
\int_{\Om_b}\chi_{D} v_0^2 \dd x=\pi \int_b^{b+1}\chi_{D_1}h_0^2 r \dd r \ge \fr{\pi \de b
}{2}\inf_{[b+\de/4,b+1-\de/4]}h_0^2
\end{align}
and
\begin{align}\label{lem85-con2}
\int_{\Om_b}v_0^2 \dd x= \pi\int_b^{b+1}h_0^2r\dd r \le 2\pi b\int_b^{b+1}h_0^2 \dd r.
\end{align}

Denote by $K$ the one dimensional compact subset $[b+\e,b+1-\e]$ of the interval $[b,b+1]$, where $\e$
is a small positive number. Now use the Harnack's inequality from \cite{SZ} applied to the equation \eqref{eq-H} in order to estimate
\begin{align*}
\sup_{K} h_0 \le C \inf_{K} h_0.
\end{align*}
Moreover, using Lemma 2.3 in \cite{BWZ} for equation \eqref{eq-h1}, with the estimates from \Cref{lem-estiB} and \Cref{lem-tau},
\begin{align*}
\norm{h_0}_{L^\infty((b,b+1))}\le C \norm{h_0}_{L^2((b,b+1))},
\end{align*}
for some $C$ independent of $b$.
Therefore, we have
\begin{align*}
\int_b^{b+1}h_0^2\dd r&=\int_K h_0^2\dd r+\int_{[b,b+1]\setminus K}h_0^2\dd r\le |K|\sup_K
h_0^2+(1-|K|)\sup_{(b,b+1)}h_0^2\\
&\le C \left(\inf_K h_0^2 + (1-|K|)\int_b^{b+1}h_0^2\dd r\right).
\end{align*}
By taking sufficiently small $\e$, we can have $C(1-|K|)\le \2$ so that we finally arrive to
\begin{align}\label{lem85-con3}
\int_b^{b+1}h_0^2\dd r \le C \inf_K h_0^2.
\end{align}
Again, we may take small $\e$ satisfying $[b+\de/4,b+1-\de/4] \subset K$. Now the conclusion follows
from \eqref{lem85-con1}, \eqref{lem85-con2}, and \eqref{lem85-con3}.
\end{proof}

\begin{proof}[Proof of Claim 2.]
The conclusion follows directly  from \Cref{lem-symmrelation} and \Cref{lem-alphaqubdd}.
\end{proof}

\noindent\textbf{Acknowledgements.}  M.d.M. Gonz\'alez is supported by the Spanish government grant
MTM2017-85757-P. Taehun Lee was supported by National Research Foundation of Korea Grant funded by the Korean Government (NRF-2014H1A2A1018664).
Ki-Ahm Lee is supported by the National Research Foundation of Korea (NRF) grant : NRF-2020R1A2C1A01006256. Ki- Ahm Lee also holds a joint appointment with the Research Institute of Mathematics of Seoul National University.

\bibliography{NCM}
\bibliographystyle{acm}

\end{document}